\documentclass[12pt]{article}%
\usepackage[colorlinks,bookmarks=true]{hyperref}
\usepackage{amsmath}
\usepackage{graphicx}
\usepackage{amsfonts}
\usepackage{amssymb}%
\usepackage{amsmath}
\usepackage{float}
\usepackage{caption}
\usepackage{subfigure}
\usepackage{graphicx}
\usepackage{epsf} 
\usepackage{epstopdf}
\usepackage{amssymb}
\usepackage{amsthm}
\usepackage{setspace}
\usepackage{mathrsfs}
\usepackage{amsfonts}
\usepackage{galois}
\usepackage{epic}
\providecommand{\U}[1]{\protect\rule{.1in}{.1in}}

\def\s{\sigma}

\def\D{\Delta}

\def\v{\varepsilon}

\def\bbb{\begin{eqnarray*}}

\def\eee{\end{eqnarray*}}

\newtheorem{theorem}{Theorem} [section]

\newtheorem{corollary}{Corollary}[section]

\newtheorem{definition}{Definition} [section]

\newtheorem{lemma}{Lemma}[section]

\newtheorem{proposition}{Proposition} [section]
\newtheorem{remark}{Remark} [section]

\numberwithin{equation}{section}

\begin{document}

\title{Singularity of the $n$-th eigenvalue of  high dimensional Sturm-Liouville problems}
\author{Xijun Hu$^\dag$, Lei Liu$^\dag$, Li Wu$^\dag$ and Hao Zhu$^\ddag$\footnote{Corresponding author. }\\[4pt]
$^\dag$Department of Mathematics, Shandong University\\
                  Jinan, Shandong 250100, P. R. China
		\\[4pt]
					  $^\ddag$Chern Institute of Mathematics, Nankai University\\[2pt]
					  Tianjin, 300071, P. R. China
 					 }
\date{}
\maketitle

\begin{abstract}
It is natural to consider continuous dependence of the $n$-th eigenvalue on $d$-dimensional ($d\geq2$) Sturm-Liouville problems  after the results on $1$-dimensional case by Kong, Wu and Zettl \cite{Kong1999}. In this paper, we find all the boundary conditions such that the $n$-th eigenvalue is not continuous,  and give complete characterization of  asymptotic behavior of the $n$-th eigenvalue. This renders a precise description of the jump phenomena of the $n$-th eigenvalue near such a boundary condition. Furthermore, we divide the space of boundary conditions into $2d+1$ layers and  show that the $n$-th eigenvalue is continuously dependent  on Sturm-Liouville equations and on boundary conditions when restricted into  each layer. In addition, we prove that the analytic and geometric multiplicities of an eigenvalue are equal.  Finally, we obtain derivative formula and  positive direction of eigenvalues with respect to boundary conditions.
\end{abstract}


\bigskip

\noindent
{\bf Contents}
\medskip

\noindent
1\ \ Introduction  \dotfill 2\\[1.4ex]
2\ \ Space of Sturm-Liouville problems  \dotfill 7\\[1.4ex]
3\ \ Basic properties of eigenvalues \dotfill 11\\[1.4ex]
4\ \ Analysis on $1$-dimensional results \dotfill 13\\[1.4ex]
5\ \ Equality of multiplicities of an eigenvalue \dotfill 17\\[1.4ex]
6\ \ Continuity of the $n$-th eigenvalue \dotfill 19\\[1.4ex]
7\ \ Singularity of the $n$-th eigenvalue \dotfill 22\\[1.4ex]
8\ \ Derivative formulas and comparison  of  eigenvalues \dotfill 33\\[1.4ex]
Acknowledgements \dotfill 35\\[1.4ex]
References\dotfill 36

\section{Introduction}
In this paper, we consider the  $d$-dimensional Sturm-Liouville equation
\begin{equation}\label{SL-equation}
-(Py')'+Qy=\lambda Wy, \;\;on\;\;[a,b],
\end{equation}
where $d\geq2$, $\lambda$ is the spectral parameter, and $P, Q$ and $W$ satisfy the following assumptions:

\noindent{\bf Hypothesis 1}:
\begin{itemize}
\item[\rm{(i)}]
$P, Q$ and $W$ are  $d\times d$-matrix symmetric  real-valued functions on $[a,b]$,
\item[\rm{(ii)}]  $P(t)\geq \mu_1$ and $W(t)\geq \mu_2$  a.e. on $t\in [a,b]$ for some $\mu_1,\mu_2>0$,
\item[\rm{(iii)}] $P, Q,  W \in L^\infty([a,b],\mathbb{R}^{d\times d})$.
\end{itemize}
The self-adjoint boundary condition is given by
\begin{equation}\label{boundary condition}
(A\; | \;B) Y(a,b) =0,
\end{equation}
where $Y(a,b)= ( -y(a)^T,
y(b)^T, (Py')(a)^T,
(Py')(b)^T)^T$, $A$ and $B$ are
$2d\times2d$ complex  matrices such that
\begin{equation} \label{self-adjoint-BC}
 {\rm rank}(A\;|\; B)=2d,\;\; AB^*=BA^*,
\end{equation}
$A^*$ is the complex conjugate transpose of $A$.
The spectrum of the Sturm-Liouville problem is bounded from below and  consists of discrete eigenvalues, which are ordered in the following
non-decreasing sequence
$$\lambda_1\leq\lambda_2\leq\lambda_3\leq\cdots\leq\lambda_n\leq \cdots,$$
with $\lambda_n\to\infty$  as $n\to\infty$, counting repeatedly according to their analytic multiplicities.

When studying the perturbation of  eigenvalues of Sturm-Liouville  problem (\ref{SL-equation})--(\ref{boundary condition}), the $n$-th eigenvalue  changes as coefficients in the equation (\ref{SL-equation}) or the boundary condition (\ref{boundary condition}) are subjected to perturbations.   The indices (i.e. $n$) of eigenvalues may change drastically in a continuous eigenvalue branch due  to the high dimension, and the eigenvalues with the same index may jump from one to another  branch in a complex way. The jump  phenomena of the $n$-th eigenvalue in high dimension are more  interesting from  geometric aspects and more  complicated than $1$-dimensional Sturm-Liouville problems. Moreover,  in computing  eigenvalues, their indices  are in general   unknown and still need to be determined due to   the importance of  the first few eigenvalues in  physical models. The high dimension, however, leads oscillation theory of  $1$-dimensional Sturm-Liouville problems to becoming invalid    and   thus makes a difficult task in numerical calculation. So the question   ``{\it what the singular (discontinuity) set of the $n$-th eigenvalue is in high dimensional  case}"  not only has strong motivation from physics and  numerical analysis, but also is  theoretical challenging. Indeed, after the previous work in $1$-dimensional case \cite{Kong1999}, it has been an open problem for several years.   We shall solve it in this paper and  give complete characterization of  asymptotic behavior of the $n$-th eigenvalue.

This question is completely answered for 1-dimensional Sturm-Liouville problems. The first breakthrough is due to Rellich \cite{Rellich1942}.  In 1950, he   gave an  example for the following 1-dimensional Sturm-Liouville problem at the ICM \cite{Rellich1950}:
\begin{align*}
-u''=\lambda u,\;\;\;\;on\;\;[0,1],
\end{align*}
 with the boundary condition
$u(0)=0,\kappa u'(1)=u(1),$
where $\kappa\in \mathbb{R}$. Though the $n$-th eigenvalue is  continuous near $\kappa=0$ from the left direction for each $n\geq1$, it is discontinuous from the right direction, and  has the following asymptotic behavior:
\begin{align*}
\lim\limits_{\kappa\to0^+}\lambda_1(\kappa)=-\infty,\;\;
\lim\limits_{\kappa\to0^+}\lambda_n(\kappa)=\lambda_{n-1}(0),\;\; n \geq 2.
\end{align*}
See also Figure 1 in P. 292 of  \cite{Kato1984}.
In 1997, Everitt et al. in \cite{Everitt1997} investigated 1-dimensional Sturm-Liouville equation (\ref{SL-equation})
with separated boundary condition
$
\cos\alpha y(a)-\sin\alpha (py')(a)=0,\cos\beta y(b)-\sin\beta (py')(b)=0.
$
By using Pr\"{u}fer transformation to (\ref{SL-equation}),
they obtained that the $n$-th eigenvalue $\lambda_n$ is continuous on $\alpha\times\beta\in[0,\pi)\times(0,\pi]$ for each $n\geq1$, and moreover, for any fixed $\alpha\in[0,\pi)$,
\begin{align*}
\lim\limits_{\beta\to0^+}\lambda_1(\alpha,\beta)=-\infty,\;\;\lim\limits_{\beta\to0^+}\lambda_n(\alpha,\beta)=\lambda_{n-1}(\alpha,\pi),\;\;n\geq2,
\end{align*}
and for any fixed $\beta\in(0,\pi]$,
\begin{align*}
\lim\limits_{\alpha\to\pi^-}\lambda_1(\alpha,\beta)=-\infty,\;\;\lim\limits_{\alpha\to\pi^-}\lambda_n(\alpha,\beta)=\lambda_{n-1}(0,\beta),\;\;n\geq2.
\end{align*}
Kong et al.  in \cite{Kong1999}  regarded the $n$-th eigenvalue as a function  on the space of Sturm-Liouville equations and that of  boundary conditions, respectively. They showed that the $n$-th eigenvalue is continuously dependent on the coefficients in (\ref{SL-equation}) for each $n\geq 1$. By using the above results in \cite{Everitt1997} and some inequalities among eigenvalues of 1-dimensional Sturm-Liouville problems obtained in \cite{Eastham1999}, they  found  the singular set of the $n$-th eigenvalue
in the space of complex (resp. real) boundary conditions.
They also gave all the asymptotic behavior of the $n$-th eigenvalue near each singular boundary condition.
See Theorems 3.39  and 3.76 in \cite{Kong1999}. Other discussions of the $n$-th eigenvalue can be found in \cite{Cao2007,Courant1953,Kong19962,Weidmann1987,Zettl2005}.

It is worthy to mention that the analytic, algebraic and geometric multiplicities of an eigenvalue are shown to be equal for 1-dimensional Sturm-Liouville problems \cite{Eastham1999,Kong2000,Kong2004,Wang2005} and some extensions \cite{Fu2012,Fu2014,Hu2011,Hu2016,Naimark1968,Shi2010}.
In particular, Kong et al. showed  the equivalence of analytic and geometric  multiplicities  even if $P$ changes sign in 1-dimensional case \cite{Kong2004}.  Naimark proved analytic and algebraic multiplicities of an eigenvalue  coincide for a class of high-order differential operators \cite{Naimark1968}.    When studying the  discontinuity of the $n$-th eigenvalue, it is always listed according to the analytic multiplicity.  From the perspective of application, however, people pay more attention to how many eigenvalue branches jump (tend to infinity) in the sense of  geometric multiplicity. So it is necessary to clarify the relationships of these multiplicities of an eigenvalue for high dimensional Sturm-Liouville problems.  Motivated by Naimark \cite{Naimark1968},  we shall  rigorously prove that the three multiplicities of an eigenvalue are the same  in high dimensional case even if $P$ is non-positive.

In this paper, we shall give the set of all the complex (or real) boundary conditions such that the $n$-th eigenvalue is not continuous in high dimensional case.  We call it to be  the singular set $\Sigma^\mathbb{C}$ (or $\Sigma^\mathbb{R}$) in the space of complex (or real) boundary conditions and call any element in $\Sigma^\mathbb{C}$ (or $\Sigma^\mathbb{R}$) to be a singular boundary condition. We mainly discuss the complex boundary conditions in this paper, and the real ones can be treated in a similar way. When there is truly difference in the discussion, we shall give remarks on providing a feasible way for the real ones.
 Our inspiration is from the symplectic geometry, especially the structure of Lagrangian-Grassmann manifold \cite{Arnold1967,Arnold2000}.  Indeed,
we shall prove that  in high dimensional case, the singular set, denoted by $\Sigma^\mathbb{C}$, consists of all the boundary conditions
 $\mathbf{A}=[A\;|\;B]$ such that
 \begin{align}\label{key-observation}
 n^0 (B)>0,
 \end{align}
where $n^0 (B)$ denotes the geometric  multiplicity of zero eigenvalue of $B$.

An accompanying  difficulty is how to give and  prove  asymptotic behavior of the $n$-th eigenvalue near a singular boundary condition.
The strategy based on the Pr\"{u}fer transformation does not work for  separated boundary conditions due to the coupling of the Sturm-Liouville equations. The inequalities argument used in \cite{Kong1999} also becomes invalid owing to the complexity of the boundary conditions, for example, the appearance of  mixing boundary conditions \cite{Wang2008}. Moreover, the directions, from which the boundary conditions tend to a more singular one, are diversified in  high dimensional case. All these make the problem  nontrivial.
  Our first task is also to study the topology of the space of boundary conditions. However,  instead of using separated, coupled and mixed boundary conditions, we choose the coordinate charts of the Lagrangian-Grassmann manifold introduced by Arnold \cite{Arnold1967,Arnold2000} to describe the space of boundary conditions.
  We divide $\Sigma^\mathbb{C}$ into $2d$ layers such that for any $[A\;|\;B]$  in the $k$-th layer $\Sigma_k^\mathbb{C}$,
\begin{align}\label{k-th layer criteria}
n^0(B)=k,
\end{align}
where $1\leq k\leq 2d$. Define $\Sigma_0^\mathbb{C}$ to be the complementary set  of $\Sigma^\mathbb{C}$ in the space of boundary conditions.
We then prove that the $n$-th eigenvalue  is continuously dependent on the Sturm-Liouville equations  and on the boundary conditions when restricted into $\Sigma_k^\mathbb{C}$ for each $n\geq 1$, where $0\leq k\leq 2d$.
 In the proof of  asymptotic behavior, besides  using the above results and the locally uniform property of eigenvalues (Lemma \ref{local continuity of eigenvalues}), our  technique is to construct various paths in different parts of the $k$-th layer in the space of boundary conditions. The asymptotic behavior of the $n$-th eigenvalue is proved for the following targets  via a step-by-step procedure:
\begin{itemize}
\item[{\rm 1.}] A non-coupled Sturm-Liouville equation and a separated  boundary condition in a path connected component of $\Sigma_k^\mathbb{C}$, see Lemmas  \ref{one path from 1-dim} and \ref{asymptotic behavoir for tilde A0}.
 \item[{\rm 2.}] The non-coupled Sturm-Liouville equation and any boundary condition in the path connected component of $\Sigma_k^\mathbb{C}$, see Lemma \ref{asymptotic behavior A0}.
 \item[{\rm 3.}] Any fixed Sturm-Liouville equation and any boundary condition in the path connected component of $\Sigma_k^\mathbb{C}$, see Lemma \ref{asymptotic behavior C0 any sturm liouville equation}.
  \end{itemize}
Our complete characterization of asymptotic behavior of the $n$-th eigenvalue near $\mathbf{A}\in\Sigma_k^\mathbb{C}$ (or $(\pmb\omega,\mathbf{A})$) is given in Theorems \ref{main theorem} and \ref{pmb-omega-mathbf{A}}, where $1\leq k\leq 2d$.
The essential characterization is that there exists a neighborhood $U$   of $\mathbf{A}$ in the whole space of boundary conditions such that $U=\cup_{0\leq i\leq k} U^i$,  and  for each $0\leq i\leq k$, $U^i\neq\emptyset$ and
\begin{align}\label{U-i1}
\lim_{U^i\ni\mathbf{B}\to\mathbf{A}}\lambda_{n}(\mathbf{B})&=-\infty,\;\;1\leq n\leq i,\\\label{U-i2}
\lim_{U^i\ni\mathbf{B}\to\mathbf{A}}\lambda_{n}(\mathbf{B})&=\lambda_{n-i}(\mathbf{A}),\;\;n> i.
\end{align}
In order to make the results explicitly, we  clarify what  $U^i$ is in Corollary \ref{corollary-u-i}.

In a forthcoming paper,
we shall show that   $[-\Psi_\lambda\;|\;\Phi_\lambda]$ defined in  (\ref{def-Philambda})--(\ref{def-Psilambda}) tends to $[I_{2d}\;|\;0_{2d}]$ as $\lambda\to -\infty$ for a more general Sturm-Liouville system. Based on this result and the theory of  Maslov index, we shall give a new proof of the  discontinuity of the $n$-th eigenvalue. Furthermore, we determine the range of the $n$-th eigenvalue not only on the whole space of boundary conditions but also on the $k$-th layer $\Sigma^\mathbb{C}_k$ (or $\Sigma^\mathbb{R}_k$), where $0\leq k\leq 2d$.

The rest of this paper is organized  as follows. In Section \ref{The space},  topology on  space of Sturm-Liouville equations, and that on  space  of complex boundary conditions are presented. Basic properties of eigenvalues are given in Section \ref{Basic  Sturm-Liouville theory}
and further analysis on $1$-dimensional results is provided in Section  \ref{Analysis on $1$-dimensional results}.
 In Section \ref{Equality of multiplicities of an eigenvalue}, it is proved that the analytic, algebraic and geometric multiplicities of an eigenvalue are equal.
Section \ref{Continuity of the $n$-th eigenvalue in 1-dimensional case}
 is devoted to proving that the $n$-th eigenvalue is continuous on the space of  Sturm-Liouville equaitons, and on each  layer in the space of boundary conditions. Singularity of the $n$-th eigenvalue  is completely characterized  in  Section \ref{Singularity of the $n$-th eigenvalue of  high dimensional Sturm-Liouville problems}. In the last section,   derivative formula  of  eigenvalues with respect to the boundary conditions and  comparison of eigenvalues   are obtained.

\bigskip

\section{Space of Sturm-Liouville problems}\label{The space}

In this section, we  introduce the topology on  space of Sturm-Liouville equations, and that on  space  of complex boundary conditions, respectively.

The space of  Sturm-Liouville equations is
\begin{align*}
\Omega:=\{(P,Q,W): P,Q\;{\rm and}\;W \;{\rm satisfy} \;{\bf{\rm Hypothesis\; 1}}\}
\end{align*}
 with product topology induced by $L^{\infty}([a,b],\mathbb{R}^{d\times d})$.
Following \cite{Kong1999}, we use bold faced (lower case) Greek letters, such as $\pmb{\omega}$, to stand for elements in $\Omega$.

Two linear algebraic systems
\begin{align*}
(A\; | \;B)Y(a,b) =0,\;\;(C\; | \;D)Y(a,b)=0,
\end{align*}
represent the same complex boundary condition if and only if there exists a matrix $T\in GL(2d,\mathbb{C})$ such that
 \begin{align*}
 (C\;|\;D)=(TA\;|\;TB),
 \end{align*}
where
$GL(2d,\mathbb{C}):=\{2d\times2d\; {\rm complex \;matrix}\; T:{\rm det}\; T \neq0\}$. Since each  boundary condition (\ref{boundary condition}) considered in this paper is self-adjoint, it must satisfy (\ref{self-adjoint-BC}).
So it is natural to take
the quotient space
\begin{align*}
\mathcal B^{\mathbb{C}} :=\lower3pt\hbox{${\rm
GL}(2d,\mathbb C)$}\backslash \raise2pt\hbox{$\mathcal{L}_{2d,4d}(\mathbb C)$},
\end{align*}
equipped with the quotient topology,  as the space of complex boundary conditions, where
$
\mathcal{L}_{2d,4d}(\mathbb{C})
:=\{2d\times4d \;{\rm complex\; matrix\; (A\;|\;B)}:\;{\rm rank}(A\;|\;B)=2d,\;AB^*=BA^*\}.$
The topology on $\mathcal{L}_{2d,4d}(\mathbb{C})$ here is induced by that of $\mathbb{C}^{8d^2}$.
The boundary condition represented by (\ref{boundary condition}) is denoted by $[A\,|\,B]:=\{(TA\;|\;TB):T\in GL(2d,\mathbb{C})\}$. Bold faced capital Latin letters, such as $\mathbf{A}$, are also used  for  boundary conditions.
See also \cite{Kong2000} in $1$-dimensional case.
$\mathcal B^{\mathbb{C}}$ coincides with the complex  Lagrangian-Grassmann manifold $\Lambda(2d)$ \cite{Arnold2000}.

Motivated by Arnold \cite{Arnold1967,Arnold2000}, we shall give the canonical atlas of local coordinate systems on $\mathcal{B}^\mathbb{C}$ in our framework.
 Let  $\begin{array}{llll}e_i=(0,\cdots,0,&1,&0,\cdots, 0)_{1\times2d}^T\\&i&\end{array}$ for any $1\leq i\leq 2d$, $S$ be a $2d\times2d$  matrix whose entries and columns are denoted by $s_{lj}$, $1\leq l, j\leq 2d$, and $s_j=(s_{1j},\cdots,s_{2dj})^T$,  respectively. Set $K$ be a subset of $\{1,2,\cdots, 2d\}$ with $\sharp(K)$ to be the number of elements in $K$.   Define

 \begin{align}\label{coordinate}
\mathcal{O}^{\mathbb{C}}_K=
 \big\{\mathbf{A}=[A\;|\;B]=&[(a_1,a_2,\cdots,a_{2d})\;|\;(b_1,b_2,\cdots,b_{2d})]:\big.\\\nonumber
a_{i}=&\begin{cases} -e_i & {\rm if}\;\;  i\in K, \\
                      s_i &  {\rm if}\;\; i\in\{1,2,\cdots,2d\}\backslash K, \end{cases}\\\nonumber
 b_{i}=&\begin{cases}  s_i & \;\;\;{\rm if}\;\;  i\in K, \\
                      e_i &  \;\;\;{\rm if}\;\; i\in\{1,2,\cdots,2d\}\backslash K,\end{cases}\\\nonumber
\big.&{\rm for\;any}\;2d\times2d\;{\rm Hermitian\;matrix\;} S=(s_1,s_2,\cdots, s_{2d}) \big\}
\end{align}
 for any $K\subset\{1,2,\cdots, 2d\}$. Below $S$ is written as $S(\mathbf{A})$  when it is necessary to indicate its dependence on $\mathbf{A}$.
 The following result gives the topology and geometric structure on $\mathcal{B}^\mathbb{C}$.\medskip

\begin{theorem} \label{structure of space of boundary conditions}
\begin{align*}
\mathcal{B}^{\mathbb{C}}=\bigcup_{K\subset\{1,2,\cdots,2d\}}\mathcal{O}^\mathbb{C}_K.
\end{align*}
Moreover, $\mathcal{B}^\mathbb{C}$ is a connected and compact real-analytic manifold of dimension $4d^2$.
\end{theorem}
\begin{remark}
Similar result is true for $\mathcal{B}^{\mathbb{R}}$ with $\mathbb{C}$ replaced by $\mathbb{R}$ except that the dimension of $\mathcal{B}^\mathbb{R}$ is $d(2d+1)$ as a real-analytic manifold.
\end{remark}

\begin{proof}
Firstly, we show that
\begin{eqnarray}
\bigcup_{K\subset\{1,2,\cdots,2d\}}\mathcal{O}^\mathbb{C}_K\subset \mathcal{B}^\mathbb{C}.\label{equ step 1}
\end{eqnarray}

For any $K\subset\{1,\cdots,2d\}$, choose $[A\;|\;B]=[(a_1,\cdots,a_{2d})\;|\;(b_1,\cdots,b_{2d})]\subset\mathcal{O}^\mathbb{C}_K$, where $$a_j=(a_{1j},\cdots,a_{2dj})^T,b_j=(b_{1j},\cdots,b_{2dj})^T, \ j=1,\cdots,2d.$$ Define
\begin{eqnarray*}
\delta_{i,K}:=\left\{
\begin{aligned}
&1,\ i\in K,\\ &0,\ i\notin K.
\end{aligned}\right.
\end{eqnarray*}
For any $1\leq i,k\leq 2d$, by (\ref{coordinate}) we have
\begin{align}\label{AKBKBKAK1}
(AB^*)_{ik}&=\sum^{2d}_{j=1}a_{ij}\bar{b}_{kj}\\\nonumber
&=\sum_{j\in K}a_{ij}\bar{b}_{kj}+\sum_{j\notin K}a_{ij}\bar{b}_{kj}\\\nonumber
&=\sum_{j\in K}-\delta_{ij}\bar{s}_{kj}+\sum_{j\notin K}s_{ij}\delta_{kj}\\\nonumber
&=-\delta_{i,K}\bar{s}_{ki}+(1-\delta_{k,K})s_{ik},
\end{align}
\begin{align}\label{AKBKBKAK2}
(BA^*)_{ik}&=\sum^{2d}_{j=1}b_{ij}\bar{a}_{kj}\\\nonumber
&=\sum_{j\in K}b_{ij}\bar{a}_{kj}+\sum_{j\notin K}b_{ij}\bar{a}_{kj}\\\nonumber
&=\sum_{j\in K}-s_{ij}\delta_{kj}+\sum_{j\notin K}\delta_{ij}\bar{s}_{kj}\\\nonumber
&=-\delta_{k,K}s_{ik}+(1-\delta_{i,K})\bar{s}_{ki}.
\end{align}
Since $S=S^*$, we have $AB^*=BA^*$. Then (\ref{equ step 1}) follows.

Conversely, let
$$[A\;|\;B]=[(a_1,\cdots,a_{2d})\;|\;(b_1,\cdots,b_{2d})]\in \mathcal{B}^{\mathbb{C}}$$ and $m_0:=\mathrm{rank}{A}$.
Suppose  $\mathrm{rank}$ $(a_{k_1},\cdots,a_{k_{m_0}})=m_0$. Denote $K:=\{k_1,\cdots,k_{m_0}\}$ and $\{1,2,\cdots,2d\}\setminus K:=\{k_{m_0+1},\cdots,k_{2d}\}.$

\textbf{Claim}: $(a_{k_1},\cdots,a_{k_{m_0}},b_{k_{{m_0}+1}},\cdots,b_{k_{2d}})$ is non-degenerate.

Firstly, we can choose $T\in GL(2d,\mathbb{C})$  such that
$$T(A\;|\;B)=\begin{pmatrix}A_1&B_1\\0&B_2\end{pmatrix}\in\mathcal{L}_{2d,4d}(\mathbb{C}),\ A_1,B_1\in \mathcal{M}_{m_0\times2d},\  B_2\in\mathcal{M}_{(2d-m_{0})\times 2d},$$
where $\mathcal{M}_{m_0\times2d}$ denotes the set of all $m_0\times 2d$ complex matrices.
Let
 $$A_1:=(a_{1,1},\cdots,a_{1,2d}),\;B_2:=(b_{2,1},\cdots,b_{2,2d}), $$
where $a_{1,l}=(a_{1,1l},\cdots,a_{1,m_0l})^T$
and $b_{2,l}=(b_{2,1l},\cdots,b_{2,(2d-m_0)l})^T,$ $1\leq l\leq 2d$. Since $[A\;|\;B]\in\mathcal{B}^\mathbb{C}$,
direct computation shows that
\begin{eqnarray*}
\begin{pmatrix}A_1\\0\end{pmatrix}\begin{pmatrix}B^*_1&B^*_2\end{pmatrix}=\begin{pmatrix}B_1\\B_2\end{pmatrix}\begin{pmatrix}A^*_1&0\end{pmatrix}\\
\Leftrightarrow\ A_1B_1^*\ {\rm symmetric}\ {\rm and}\ A_1B_2^*=0.
\end{eqnarray*}
 Let
 \begin{align*}
 E&:=(E_1\;|\;E_2)=(a_{1,k_1},\cdots,a_{1,k_{m_0}}\;|\; a_{1,k_{m_0+1}},\cdots,a_{1,k_{2d}}),\\
 F&:=(F_1\;|\;F_2)=( b_{2,k_1},\cdots,b_{2,k_{m_0}}\;|\;b_{2,k_{m_0+1}},\cdots,b_{2,k_{2d}}).
 \end{align*}
 Note that ${\rm rank} E_1=\mathrm{rank}E=\mathrm{rank}A_1=m_0$ and $\mathrm{rank}F=\mathrm{rank}B_2=2d-m_0$.
Since $A_1B^*_2=0$, we have $EF^*=E_1F^*_1+E_2F^*_2=0$.
Direct calculation gives
$$(F_2-F_1E_1^{-1}E_2)F_2^*=F_2F^*_2-F_1E_1^{-1}E_2F^*_2=F_2F^*_2+F_1E_1^{-1}E_1F^*_1=FF^*,$$
which yields
$$2d-m_0\geq\mathrm{rank}F^*_2\geq\mathrm{rank}FF^*=\mathrm{rank}F=2d-m_0.$$
 Then $\mathrm{rank}F_2=2d-m_0.$
 Thus
$$(a_{k_1}\cdots,a_{k_{m_0}},b_{k_{m_0+1}},\cdots,b_{2d})=T^{-1}\begin{pmatrix}E_1&*\\0& F_2\end{pmatrix}$$
 is non-degenerate and this claim holds.

Let
$$T_1:=(-a_{k_1},\cdots,-a_{k_{m_0}},b_{k_{m_0+1}},\cdots,b_{k_{2d}}).$$
Then $$(T^{-1}_1A\;|\;T_1^{-1}B)=(a'_{1},\cdots,a'_{2d},b'_{1},\cdots,b'_{2d}),$$
where $a'_{k_i}=-e_i,\ i=1,\cdots,m_0,\ b'_{k_j}=e_j,\ j=m_0+1,\cdots,2d$.
Let $S=(s_1,\cdots,s_{2d})=(s_{ik})\in\mathcal{M}_{2d\times 2d}$ and
\begin{eqnarray*}
s_l=\left\{
\begin{aligned}
a'_{l},&\;\; l\notin K,\\ b'_{l},&\;\; l\in K.
\end{aligned}\right.
\end{eqnarray*}
  Since $AB^*=BA^*$, by the calculation in (\ref{AKBKBKAK1})--(\ref{AKBKBKAK2}), we have $s_{ik}=\bar{s}_{ki}$. Thus  $[A\;|\;B]\subset\mathcal{O}_K^\mathbb{C}$. Other assertions are direct consequences of properties of Lagrangian-Grassmann manifold \cite{Arnold1967,Arnold2000}. This completes the proof.
\end{proof}

The product space $\Omega\times\mathcal{B}^{\mathbb{C}}$ is the space of   Sturm-Liouville problems, and $(\pmb{\omega},\mathbf{A})$  is used to stand for an element in $\Omega\times\mathcal{B}^{\mathbb{C}}$.

\section{Basic properties of eigenvalues}\label{Basic  Sturm-Liouville theory}

Firstly, we introduce the following weighted space:
\begin{align*}
L^2_W([a,b],\mathbb{C}^d)=\{y:\int_a^by(t)^*W(t)y(t)dt<+\infty\},
\end{align*}
with the inner product $\langle y,z\rangle_W=\int_a^bz(t)^*W(t)y(t)dt$.

Let $(\pmb\omega,\mathbf{A})\in\Omega\times\mathcal{B}^\mathbb{C}$ with $\pmb\omega=(P,Q,W)$ and $\mathbf{A}=[A\;|\;B]$. Then the corresponding  Sturm-Liouville operator
\begin{align}\label{Sturm-Liouville operator}
T_{(\pmb\omega,\mathbf{A})}y=W^{-1}(-(Py')'+Qy)
\end{align}
is self-adjoint with the domain
\begin{align}\label{domain}
D_{\mathbf{A}}=\{y\in L^2_W([a,b],\mathbb{C}^d):&y, Py'\in AC([a,b],\mathbb{C}^d),\\\nonumber
& T_{(\pmb\omega,\mathbf{A})}y\in L^2_W([a,b],\mathbb{C}^d), y\; {\rm satisfies}\;  (\ref{boundary condition})\}.
\end{align}

For any $\lambda\in\mathbb{C}$, let
 $\phi_{1,\lambda},\cdots,\phi_{2d,\lambda}$
 be the fundamental  solutions to (\ref{SL-equation}) determined by the initial conditions
\begin{align*}
\begin{pmatrix}\phi_{1,\lambda}(a)&\cdots&\phi_{2d,\lambda}(a)\\
P\phi_{1,\lambda}'(a)&\cdots&P\phi_{2d,\lambda}'(a)\end{pmatrix}=I_{2d}.
\end{align*}
Denote
\begin{align}\label{def-Philambda}
\Phi_\lambda:=
\begin{pmatrix}-\phi_{1,\lambda}(a)&\cdots&-\phi_{2d,\lambda}(a)\\
\phi_{1,\lambda}(b)&\cdots&\phi_{2d,\lambda}(b)\end{pmatrix},\\\label{def-Psilambda}
\Psi_\lambda:=
\begin{pmatrix}P\phi'_{1,\lambda}(a)&\cdots&P\phi'_{2d,\lambda}(a)\\
P\phi'_{1,\lambda}(b)&\cdots&P\phi'_{2d,\lambda}(b)\end{pmatrix}.
\end{align}
Then   $\Phi_\lambda$ and $\Psi_\lambda$  are entire $2d\times2d$-matrix valued functions of $\lambda$.

\begin{lemma}\label{fundamental lemma}
The spectrum of  the $d$-dimensional Sturm-Liouville problem $(\pmb\omega,\mathbf{A})$ consists of isolated eigenvalues, which are all real and bounded from below. Moreover, $\lambda\in\mathbb{R}$ is an eigenvalue of $(\pmb\omega,\mathbf{A})$ if and only of $\lambda$ is a zero of
\begin{align}\label{Gamma-lambda}
\Gamma_{(\pmb\omega,\mathbf{A})}(\lambda):=\det(A\Phi_\lambda+B\Psi_\lambda).
\end{align}
\end{lemma}
\begin{proof} The proof is similar to that of Lemma 4.5 in \cite{Zettl1997}.
\end{proof}

\begin{definition}
Let $\lambda$ be an eigenvalue of $(\pmb\omega,\mathbf{A})$.
The order of $\lambda$ as a zero of $\Gamma_{(\pmb\omega,\mathbf{A})}$ is called its analytic multiplicity. The number of linearly independent eigenfunctions for $\lambda$ is called its geometric multiplicity. The dimension of the space $E_{\lambda}=\{y\in L_W^2([a,b],\mathbb{C}^d):(T_{(\pmb\omega,\mathbf{A})}-\lambda)^ky=0\;{\it for\;some\;integer\;}k\geq 1\}$ is called its algebraic multiplicity.
\end{definition}

We show in Theorem \ref{ equivalence of three multiplicities of an eigenvalue} that   the three multiplicities of  an eigenvalue are equal.
Thus  we shall not distinguish them.
The following result is  locally continuous dependence of eigenvalues on  Sturm-Liouville problems, which can be proved by Rouch\'{e}'s Theorem \cite{Conway1986}. See  \cite{Kong1999,Zhu2017} in $1$-dimensional case.

\begin{lemma} \label{local continuity of eigenvalues}
Let $r_1<r_2$ be two real numbers such that neither of them is an eigenvalue of a given Sturm-Liouville problem $(\pmb\omega,\mathbf{A}),$ and $n\geq 0$ be the number of eigenvalues of  $(\pmb\omega,\mathbf{A})$ in the interval $(r_1,r_2)$. Then there exists a neighborhood $U$ of $(\pmb\omega,\mathbf{A})$ in $\Omega\times\mathcal{B}^{\mathbb{C}}$ such that each $(\pmb\sigma,\mathbf{B})\in U$ has exactly $n$ eigenvalues in $(r_1,r_2)$, and  neither $r_1$ nor $r_2$ is an eigenvalue of $(\pmb\sigma,\mathbf{B})$.
\end{lemma}
\begin{proof}
Let  $R:=\{z\in\mathbb{C}:|z-(r_1+r_2)/2|<(r_2-r_1)/2\}$ and
$\eta:=\min\limits_{\lambda\in \partial R}|\Gamma_{(\pmb\omega,\mathbf{A})}(\lambda)|,$
where $\partial R$ denotes the boundary of $R$.
Then $\eta>0$ by Lemma \ref{fundamental lemma}.  By the compactness of $\partial R$ and the uniform continuity of $\Gamma_{(\pmb\omega,\mathbf{A})}$ on $(\pmb\omega,\mathbf{A})$ and $\lambda$, there exists a neighborhood $U$ of $(\pmb\omega,\mathbf{A})$  in $\Omega\times\mathcal{B}^\mathbb{C}$ such that $|\Gamma_{(\pmb\sigma,\mathbf{B})}(\lambda)-\Gamma_{(\pmb\omega,\mathbf{A})}(\lambda)|<\eta$ for all $\lambda\in\partial R$ and for all $(\pmb\sigma,\mathbf{B})\in U$,
  which also implies that $|\Gamma_{(\pmb\sigma,\mathbf{B})}(\lambda)|\geq|\Gamma_{(\pmb\omega,\mathbf{A})}(\lambda)|
-|\Gamma_{(\pmb\sigma,\mathbf{B})}(\lambda)-\Gamma_{(\pmb\omega,\mathbf{A})}(\lambda)|>0$.
Thus neither $r_1$ nor $r_2$ is an eigenvalue of $(\pmb\sigma,\mathbf{B})\in U$. Since $\Gamma_{(\pmb\sigma,\mathbf{B})}-\Gamma_{(\pmb\omega,\mathbf{A})}$ and $\Gamma_{(\pmb\omega,\mathbf{A})}$ are both entire functions of $\lambda$, $\Gamma_{(\pmb\sigma,\mathbf{B})}$ and $\Gamma_{(\pmb\omega,\mathbf{A})}$  have the same
number of zeros in $R$, counting order, by Rouche's Theorem.   The proof is complete by the fact of the reality of eigenvalues.
\end{proof}

The next result is a direct consequence of Lemma \ref{local continuity of eigenvalues}.

\begin{lemma}\label{continuous eigenvalue branch}
Let $\lambda_*$ be an eigenvalue  with multiplicity $m$ of $(\pmb\omega,\mathbf{A})$, and $r_1<r_2$ be two real numbers  such that $\lambda_*\in(r_1,r_2)$ is the only eigenvalue  of  $(\pmb\omega,\mathbf{A})$ in the interval
$[r_1, r_2]$. Then there exist a  connected neighborhood $U$ of $(\pmb\omega,\mathbf{A})$ in $\Omega\times\mathcal{B}^{\mathbb{C}}$ and continuous functions $\Lambda_1,\cdots,\Lambda_m$ defined
on $U$ such that $r_1<\Lambda_1(\pmb\sigma,\mathbf B) \leq\cdots\leq
\Lambda_m(\pmb\sigma,\mathbf B) <r_2$ for each
$(\pmb\sigma,\mathbf B) \in U$, where
$\Lambda_1(\pmb\sigma,\mathbf B),\cdots,\Lambda_m(\pmb\sigma,\mathbf B)$ are eigenvalues of
$(\pmb\sigma,\mathbf B)$.
\end{lemma}

These functions in Lemma \ref{continuous eigenvalue branch} are locally called continuous eigenvalue branches. When $m=1$, $\Lambda_1$ is called the continuous simple eigenvalue branch.
Then we shall make a continuous choice of eigenfunctions for the eigenvalues along a  continuous simple eigenvalue branch.

\begin{lemma}\label{continuous choice of eigenfunctions}
  Let $\lambda_*$ be a simple eigenvalue  of $(\pmb\omega,\mathbf{A})$,
$u_0$ be a given eigenfunction for $\lambda_*$, and
$\Lambda$ be a continuous simple eigenvalue branch defined on a  neighborhood ${U}$ of $(\pmb\omega,\mathbf{A})$
in $\Omega\times\mathcal{B}^{\mathbb{C}}$ through $\lambda_*$.  Then
 there exists a neighborhood ${U}_1\subset{U} $ of $(\pmb\omega,\mathbf{A})$ such that   for any $(\pmb\sigma,\mathbf{B})\in{U}_1$,
  there is an eigenfunction $u_{\Lambda(\pmb\sigma,\mathbf{B})}$ for $\Lambda(\pmb\sigma,\mathbf{B})$  satisfying that
 $u_{\Lambda(\pmb\omega,\mathbf{A})}=u_0$, and  $u_{\Lambda(\cdot)}$ and $pu'_{\Lambda(\cdot)}$ are continuous on ${U}_1$ in the sense that
for any $(\pmb\sigma,\mathbf{B})\in{U}_1$,
$u_{\Lambda(\pmb\tau,\mathbf{C})}\to u_{\Lambda(\pmb\sigma,\mathbf{B})}$ and  $pu'_{\Lambda(\pmb\tau,\mathbf{C})}\to pu'_{\Lambda(\pmb\sigma,\mathbf{B})}$
as ${U}_1\ni(\pmb\tau,\mathbf{C})\to(\pmb\sigma,\mathbf{B})$ both uniformly on  $[a,b]$.
\end{lemma}

\begin{proof}
The proof is similar as that of Theorem 3.1 in \cite{Kong1996}.
\end{proof}

When $\pmb\omega$ (or $\mathbf{A}$) is fixed, we can get corresponding results on a neighborhood of $\mathbf{A}$ (or $\pmb\omega$)  as those   in Lemmas \ref{local continuity of eigenvalues}--\ref{continuous choice of eigenfunctions}.
Then we turn to present the continuity principle for the $n$-th eigenvalue.

\begin{lemma}\label{continuity principle}
Let $\mathcal{O}$ be a subset of $\Omega\times\mathcal{B}^{\mathbb{C}}$. If $\lambda_1$ is bounded from below on $\mathcal{O}$, then the restriction of   the $n$-th eigenvalue to $\mathcal{O}$  is continuous for each $n\geq1$.
\end{lemma}
\begin{lemma}\label{continuity-change-indices}
If $\mathcal{O}$ is a subset of $\Omega\times \mathcal{B}^{\mathbb{C}}$, $(\pmb\omega,\mathbf{A})\notin\mathcal{O}$ is an accumulation point of $\mathcal{O}$,
\begin{align*}
\lim\limits_{\mathcal{O}\ni(\pmb\sigma,\mathbf{B})\to(\pmb\omega,\mathbf{A})}\lambda_{n}(\pmb\sigma,\mathbf{B})=-\infty
\end{align*}
for any $n=1,\cdots,m$, where $m\geq0$, and $\lambda_{m+1}$ is bounded from below on $\mathcal{O}$, then
\begin{align*}
\lim\limits_{\mathcal{O}\ni(\pmb\sigma,\mathbf{B})\to(\pmb\omega,\mathbf{A})}\lambda_{n}(\pmb\sigma,\mathbf{B})=\lambda_{n-m}(\pmb\omega,\mathbf{A})
\end{align*}
for any $n\geq m+1$.
\end{lemma}

By using Lemma \ref{local continuity of eigenvalues},
the proofs of  Lemmas \ref{continuity principle}--\ref{continuity-change-indices} are similar to those of Theorems 1.40--1.41 in \cite{Kong1999}, respectively.

\section{Analysis on $1$-dimensional results}\label{Analysis on $1$-dimensional results}

In this section, we reform the singular boundary conditions in the frame  (\ref{coordinate}) and refine the results of Theorems 3.39 and 3.76 in \cite{Kong1999}.

The explicit coordinate systems  (\ref{coordinate}) in 1-dimensional case are as follows:
\begin{align}\label{coordinate-1-dim}
\mathcal{O}_{\emptyset}^{\mathbb{C}}=&\left \{\left [\begin{array} {cccc}s_{11}&s_{12}&1&0\\
\bar{s}_{12}&s_{22}&0&1\end{array}  \right ]:\; s_{11},s_{22}\in\mathbb{R},s_{12}\in \mathbb{C}\right \},\vspace{3mm}\\\nonumber
 \mathcal{O}_{\{1\}}^{\mathbb{C}}=&\left \{\left [\begin{array} {cccc}-1&s_{12}&s_{11}&0\\
0&s_{22}&\bar{s}_{12}&1\end{array}  \right ]:\; s_{11},s_{22}\in\mathbb{R},s_{12}\in \mathbb{C}\right \},\vspace{3mm}\\\nonumber
\mathcal{O}_{\{2\}}^{\mathbb{C}}=&\left \{\left [\begin{array} {cccc}s_{11}&0&1&s_{12}\\
\bar{s}_{12}&-1&0&s_{22}\end{array}  \right ]:\; s_{11},s_{22}\in\mathbb{R},s_{12}\in \mathbb{C}\right \},\vspace{3mm}\\\nonumber
 \mathcal{O}_{\{1,2\}}^{\mathbb{C}}=&\left \{\left [\begin{array} {cccc}-1&0&s_{11}&s_{12}\\
0&-1&\bar{s}_{12}&s_{22}\end{array}  \right ]:\; s_{11},s_{22}\in\mathbb{R},s_{12}\in \mathbb{C}\right \}.
\end{align}

In order to refine the above results ,we need the following notation, which will be used for any $d\geq 1$ in the sequel. For a nonempty subset $K=\{n_1,\cdots,n_{m_0}\}\subset \{1,\cdots, 2d\}$ and for any $\mathbf{A}\in\mathcal{O}_K^\mathbb{C}$, let
\begin{align}\label{def-SK}
S_K(\mathbf{A})=\begin{pmatrix}s_{n_1n_1}&s_{n_1n_2}&\cdots&s_{n_{1}n_{m_0}}\\
\bar{s}_{n_1n_2}&{s}_{n_2n_2}&\cdots&s_{n_2n_{m_0}}\\
\vdots&\vdots&&\vdots\\
\bar{s}_{n_{1}n_{m_0}}&\bar{s}_{n_{2}n_{m_0}}&\cdots&s_{n_{m_0}n_{m_0}}\end{pmatrix}.
\end{align}
In  1-dimensional case, $S_{\{i\}}(\mathbf{A})=(s_{ii})$, $i=1,2$, and  $S_{\{1,2\}}(\mathbf{A})=\begin{pmatrix}s_{11}&s_{12}\\\bar{s}_{12}&s_{22}\end{pmatrix}$.

Let $n^-(S_K(\mathbf{A})), n^0(S_K(\mathbf{A}))$ and $n^+(S_K(\mathbf{A}))$ denote the total multiplicity of negative, zero and positive eigenvalues of $S_K(\mathbf{A})$, respectively.
For a nonempty subset $K\subset\{1,\cdots, 2d\}$, define
\begin{align}\label{def-Jn-n0n+}
&J^{(n^0,n^+,n^-)}_{\mathcal{O}_{K}^{\mathbb{C}}}\\\nonumber
:=&\{\mathbf{A}\in\mathcal{O}_{K}^{\mathbb{C}}|n^0(S_K(\mathbf{A}))=n^0,n^+(S_K(\mathbf{A}))=n^+,n^-(S_K(\mathbf{A}))=n^-\}
\end{align}
for three nonnegative integers $n^0,$ $n^+$ and $n^-$ satisfying $n^0+n^++n^-=\sharp(K)$. Then we get the following refinement from Theorem 3.76 in \cite{Kong1999}. One key fact
in the following proposition is  that $n^+(S_K(\mathbf{B}))-n^+(S_K(\mathbf{A}))$ is the number of eigenvalues which tend to $-\infty$ as $\mathbf{B}\to\mathbf{A}$.

\begin{proposition}\label{1dimrefinement}
\begin{itemize}
\item[{\rm (i)}] The restriction of  $\lambda_n$ to    $\Sigma_k^\mathbb{C}$ is continuous for each $n\geq1$, where $k=0,1,2$.
\item[{\rm (ii)}] Consider the restriction of $\lambda_n$ to $\mathcal{O}_K^{\mathbb{C}}$ for each $n\geq1$, where $K\subset\{1,2\}$.
\begin{itemize}
\item[{\rm (iia)}] The restriction of $\lambda_n$ to $\mathcal{O}_\emptyset^{\mathbb{C}}$ is continuous.
\item[{\rm (iib)}] Let $K$ be nonempty, $0\leq n^0<n^0_1\leq \sharp(K)$, $n^+\geq n_1^+$ and $n^-\geq n_1^-$. Then
for any $\mathbf{A}\in J^{(n_1^0,n_1^+,n_1^-)}_{\mathcal{O}_{K}^{\mathbb{C}}}$, we have
\begin{align*}
\lim\limits_{J^{(n^0,n^+,n^-)}_{\mathcal{O}_{K}^{\mathbb{C}}}\ni\mathbf{B}\to\mathbf{A}}\lambda_n(\mathbf{B})&=-\infty,\;n\leq n^+-n_1^+,\\
\lim\limits_{J^{(n^0,n^+,n^-)}_{\mathcal{O}_{K}^{\mathbb{C}}}\ni\mathbf{B}\to\mathbf{A}}\lambda_n(\mathbf{B})&=
\lambda_{n-(n^+-n_1^+)}(\mathbf{A}),\;n >n^+-n_1^+.
\end{align*}
\end{itemize}
\end{itemize}
\end{proposition}

\begin{remark} Inspired of \cite{Liu2018,Long1991,Long2002}, we provide an  intuitional representation of sets in the space  of real boundary conditions, which is also  helpful  to understand
the global concept of singular boundary conditions.
Let
\begin{align*}
 \mathcal{O}_{12,++}^{\mathbb{R}}=&\left \{\mathbf{A}\in\mathcal{O}_{\{1,2\}}^{\mathbb{R}}:\; s_{11}>0,s_{12}\in \mathbb{R},s_{11}s_{22}>|s_{12}|^2 \right \},\vspace{3mm}\\
\mathcal{O}_{12,--}^{\mathbb{R}}=&\left \{\mathbf{A}\in\mathcal{O}_{\{1,2\}}^{\mathbb{R}}:\; s_{11}<0,s_{12}\in \mathbb{R},s_{11}s_{22}>|s_{12}|^2 \right \},\vspace{3mm}\\
\mathcal{O}_{12,+-}^{\mathbb{R}}=&\left \{\mathbf{A}\in\mathcal{O}_{\{1,2\}}^{\mathbb{R}}:\; s_{12}\in \mathbb{R},s_{11}s_{22}<|s_{12}|^2 \right \},\vspace{3mm}\\
\mathcal{O}_{12,+0}^{\mathbb{R}}=&\left \{\mathbf{A}\in\mathcal{O}_{\{1,2\}}^{\mathbb{R}}:\; s_{11}+s_{22}>0, s_{12}\in \mathbb{R},s_{11}s_{22}=|s_{12}|^2 \right \},\vspace{3mm}\\
\mathcal{O}_{12,-0}^{\mathbb{R}}=&\left \{\mathbf{A}\in\mathcal{O}_{\{1,2\}}^{\mathbb{R}}:\; s_{11}+s_{22}<0, s_{12}\in \mathbb{R},s_{11}s_{22}=|s_{12}|^2 \right \}.
\end{align*}

Consider an element $\mathbf{A}=[-I\;|\; S]$ in the coordinate chart $(\mathcal{O}^{\mathbb{R}}_{\{1,2\}},\phi_{12})$, where $\phi_{12}:(\mathbf{A}\rightarrow S:=\begin{pmatrix}s_{11}&s_{12}\\s_{12}&s_{22}\end{pmatrix})$.
Firstly we define a map
\begin{align*}
\mathrm{rep}_{12}=\tilde{\Delta}^{-1}\circ\phi_{12}^{-1}: \{S|s_{11},s_{12},s_{22}\in \mathbb{R}\}&\rightarrow (D^1\times S^1)/f,\\
S&\mapsto[(r,z,\theta)],
\end{align*}
where
\begin{align*}
[(r,z,\theta)]\in(D^1\times S^1)/f=&\{(r,z,\theta)|(r-2)^2+z^2<1\}\cup\\
&\{\{(r,z,\theta),(r,z,\theta+\pi)\}|(r-2)^2+z^2=1\},
\end{align*}
the map $\tilde{\Delta}:(D^1\times S^1)/f\rightarrow Lag(2,\mathbb{R})\simeq\mathcal{B}^{\mathbb{R}}$ is a homeomorphism defined in the  proof of Theorem 1 in \cite{Liu2018}.

Under the map $\mathrm{rep}_{12}$, we obtain $\mathrm{rep}_{12}(0)=(2,0,0)$ and $\mathrm{rep}_{12}(-\tan\frac{\theta}{2} I_2)=(2,0,\theta)$ with $\theta\in\mathbb{R}/2\pi\mathbb{Z}$. Furthermore,
\begin{align*}
\phi_{12}^{-1}(-\tan\frac{\theta}{2} I_2)\ is\ \left\{
\begin{aligned}
&in\ \mathcal{O}^{\mathbb{R}}_{12,--},\ if\ \theta\in(0,\pi),\\
&Dirichlet\ boundary\ condition,\ if\ \theta=0,\\
&in\ \mathcal{O}^{\mathbb{R}}_{12,++},\ if\ \theta\in(-\pi,0),\\
&converging\ to\ Neumann\ boundary\ condition,\\
&\;\;\;\;\;\;\;\;\;\;\;\;\;\;\;\;\;\;\;\;\;\;\;\;\;\;\;\;\;\;\;\;\;\;\;\;\;\;\;\;\;\;\;\;\;\;\;\;\;\;\;\;\;\;\; if\ \theta\rightarrow\pm\pi.
\end{aligned}
\right.
\end{align*}

\begin{figure}
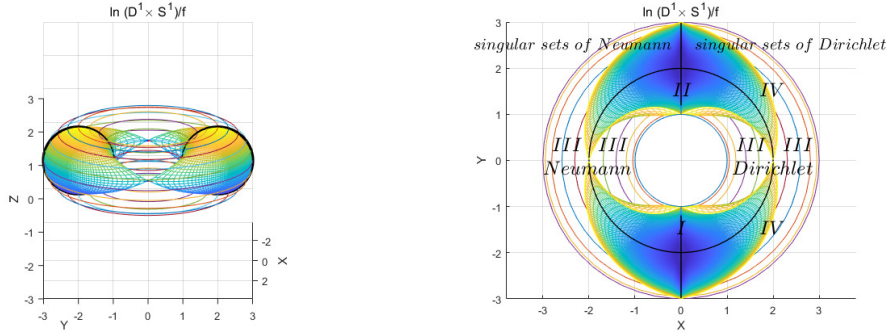

\begin{minipage}{7cm}
\includegraphics[width=6cm]{partitions-1.jpg}
\end{minipage}
\begin{minipage}{7cm}
\includegraphics[width=6cm]{partitions-2.jpg}
\end{minipage}
\caption{\small{The intuitional representation of $\mathcal{O}^{\mathbb{R}}_{12,++},\mathcal{O}^{\mathbb{R}}_{12,+-},\mathcal{O}^{\mathbb{R}}_{12,--},\mathcal{O}^{\mathbb{R}}_{12,+0},\mathcal{O}^{\mathbb{R}}_{12,-0}$.}}
\end{figure}

In Figure 1, the inner torus with two shrinking points denotes two singular cycles, i.e. the part with $x\leq0 \;(or x\geq0)$ is the set of elements which have common subspace with Neumann (or Dirichlet) boundary condition $(\mathrm{rep}_{12}(-2,0,0))$ $(or\; \mathrm{rep}_{12}(2,0,0))$. Note that inner torus lies on the outside torus if $x=0$.

Since the left part $(x\leq0)$ of the inner torus is not in the range of $\mathrm{rep}_{12}$, by the path connectedness of $\mathcal{O}^{\mathbb{R}}_{12,++},\mathcal{O}^{\mathbb{R}}_{12,+-},\mathcal{O}^{\mathbb{R}}_{12,--},\mathcal{O}^{\mathbb{R}}_{12,+0}$ and $\mathcal{O}^{\mathbb{R}}_{12,-0}$ we have\\
Region I denotes $\mathcal{O}^{\mathbb{R}}_{12,++}$, i.e. the inside of the inner torus under $Y$-axis. The curve $\mathrm{rep}_{12}(-\frac{\tan\theta}{2} I_2),\ \theta\in(-\pi,0)$  is in Region I. \\
Region II denotes $\mathcal{O}^{\mathbb{R}}_{12,--}$, i.e. the inside of the inner torus upper $Y$-axis. The curve $\mathrm{rep}_{12}(-\frac{\tan\theta}{2} I_2),\ \theta\in(0,\pi)$  is in Region II.\\
Region III denotes $\mathcal{O}^{\mathbb{R}}_{12,+-}$, i.e. the outside of the inner torus in $(D^1\times S^1)/f$.\\
Region IV denotes $\mathcal{O}^{\mathbb{R}}_{12,+0}$ and $\mathcal{O}^{\mathbb{R}}_{12,-0}$. More precisely,  the parts of the inner torus with $X>0,Y<0$ and  with $X,Y>0$ denote  $\mathcal{O}^{\mathbb{R}}_{12,+0}$ and  $\mathcal{O}^{\mathbb{R}}_{12,-0}$, respectively.
\end{remark}

\section{Equality of multiplicities of an eigenvalue}\label{Equality of multiplicities of an eigenvalue}
In this section, we  show the equivalence of the analytic, algebraic and geometric multiplicities of an eigenvalue.

\begin{theorem}\label{ equivalence of three multiplicities of an eigenvalue}
The analytic, algebraic and geometric multiplicities of an eigenvalue of $(\pmb\omega,\mathbf{A})\in\Omega\times \mathcal{B}^\mathbb{C}$ are equal.
\end{theorem}

\begin{proof}
Fix $\lambda_*$ be an eigenvalue of $(\pmb\omega,\mathbf{A})$ with $\pmb\omega=(P,Q,W)$ and $\mathbf{A}=[A\,|\,B]$. Since $T_{(\pmb\omega,\mathbf{A})}$ is self-adjoint, the algebraic and geometric multiplicities of $\lambda_*$ coincide. We shall show that the analytic and  geometric ones are the same.

By $p$ and $\kappa$ denote the geometric and analytic multiplicities of $\lambda_*$, respectively. Let $\varphi_i$, $1\leq i\leq p$, be the linearly independent eigenfunctions of $\lambda_*$. Choose $\varphi_i$, $p+1\leq i\leq 2d$ be the solutions of (\ref{SL-equation}) with $\lambda=\lambda_*$ such that $\varphi_i$, $1\leq i\leq 2d$, are linearly independent.
Let $y_i(\cdot,\lambda)$ be the solutions of (\ref{SL-equation}) with $\lambda\in\mathbb{C}$ such that $y_i(a,\lambda)=\varphi_i(a), Py'_i(a,\lambda)=P\varphi'_i(a), 1\leq i\leq 2d$.
Then $y_i(\cdot,\lambda_*)=\varphi_i$ and $y_i$ has the Taylor expansion $$y_i(x,\lambda)=\sum\limits_{j=0}^{\infty}\varphi^{(j)}_i(x)(\lambda-\lambda_*)^j$$
 with $\varphi_i^{(0)}=\varphi_i$, $1\leq i\leq 2d$.
Let $\tilde \Phi_\lambda=\begin{pmatrix}y_1&\cdots&y_{2d}\\Py_1'&\cdots&Py_{2d}'\end{pmatrix}$.
By (\ref{Gamma-lambda}), we have
\begin{align}\label{Gamma-for-multiplicity-eigenvalue}
&\Gamma_{(\pmb\omega, \mathbf{A})}(\lambda)\nonumber\\
=&\det (\tilde A+\tilde B \tilde \Phi_\lambda(b)\tilde\Phi^{-1}_\lambda(a))
=\det (\tilde A \tilde\Phi_\lambda(a)+\tilde B \tilde \Phi_\lambda(b))\det(\tilde\Phi^{-1}_\lambda(a)),
\end{align}
where
$$(\tilde A\;,\;\tilde B)=(A\;,\;B)\begin{pmatrix}-I_d&0&0&0\\0&0&I_d&0\\0&I_d&0&0\\0&0&0&I_d\end{pmatrix}.$$
Note that
$$\tilde A\begin{pmatrix}\varphi_i(a)\\P\varphi_i'(a)\end{pmatrix}+\tilde B \begin{pmatrix}\varphi_i(b)\\P\varphi_i'(b)\end{pmatrix}=0,\;1\leq i\leq p,$$
which yields that the $i$-th column of $\Gamma_{(\pmb\omega, \mathbf{A})}(\lambda)$ must contain the factor $(\lambda-\lambda_*)$. So
$$\Gamma_{(\pmb\omega, \mathbf{A})}(\lambda)=(\lambda-\lambda_*)^p\tilde \Gamma_{(\pmb\omega, \mathbf{A})}(\lambda).$$
It suffices to show that $\tilde \Gamma_{(\pmb\omega, \mathbf{A})}(\lambda_*)\neq0$.
Let
 $$\hat \Phi=\begin{pmatrix}\varphi_1^{(1)}&\cdots&\varphi_p^{(1)}&\varphi_{p+1}&\cdots&\varphi_{2d}\\
P\varphi_1^{(1)'}& \cdots& P\varphi_p^{(1)'}&P\varphi'_{p+1}&\cdots& P\varphi'_{2d}\end{pmatrix}.$$
Then
\begin{align}\label{Gamma-for-multiplicity-eigenvalue2}
\tilde \Gamma_{(\pmb\omega, \mathbf{A})}(\lambda_*)=\det(\tilde A \hat\Phi(a)+\tilde B\hat \Phi(b))\det(\tilde\Phi^{-1}_{\lambda_*}(a)).
\end{align}
Suppose $\tilde \Gamma_{(\pmb\omega, \mathbf{A})}(\lambda_*)=0$. By $\chi_1,\cdots,\chi_{2d}$ denote the columns of $\tilde A \hat\Phi(a)+\tilde B\hat \Phi(b)$. Then
\begin{align}\label{Contradiction-for-multiplicity-eigenvalue1}
\sum\limits_{i=1}^{2d}c_i\chi_i=0
\end{align}
for   $c_{1},\cdots,c_{2d}\in\mathbb{C}$ to be not all vanished. We divide the discussion into two cases below.

Case 1. $c_1=\cdots=c_p=0$.

Let $\psi=\sum\limits_{i=p+1}^{2d}c_i\varphi_i$. Then $\psi$ is a nontrivial solution of (\ref{SL-equation}) since $\varphi_{p+1},\cdots,$ $\varphi_{2d}$ are linearly independent solutions of (\ref{SL-equation}) with $\lambda=\lambda_*$.
By (\ref{Contradiction-for-multiplicity-eigenvalue1}), we get
\begin{align*}\tilde A \begin{pmatrix}\psi(a)\\P\psi'(a)\end{pmatrix}+\tilde B \begin{pmatrix} \psi(b)\\ P\psi'(b)\end{pmatrix}=0,\end{align*}
which implies that $\psi$ is an eigenfunction for $\lambda_*$. Thus
$\psi=\sum\limits_{i=1}^{p}d_i\varphi_i=\sum\limits_{i=p+1}^{2d}c_i\varphi_i$ with $d_1,\cdots,d_p$ to be not all zero, which is a contradiction since $\varphi_1,\cdots,\varphi_{2d}$ are linearly  independent.

Case 2. $c_1,\cdots,c_p$ are not all vanished.

Let
$$\tilde y=\sum\limits_{i=1}^pc_iy_i+\sum\limits_{i=p+1}^{2d}c_{i}(\lambda-\lambda_*)y_i.
$$
Then $\tilde y$ is a nontrivial solution of
$$-(P\tilde y')'+Q\tilde y=\lambda W \tilde y.$$
It is obvious that $\tilde y(\cdot,\lambda_*)$ satisfies the boundary condition $\mathbf{A}$. Differentiating the above equation by $\lambda$, we get
\begin{align}\label{Differentiating-equation}
-(P({\partial \tilde y\over \partial \lambda})')'+Q{\partial \tilde y\over \partial \lambda}=W\tilde y+\lambda W{\partial \tilde y\over \partial \lambda}.
\end{align}
Let $\lambda=\lambda_*$. Then ${\partial \tilde y\over \partial \lambda}|_{\lambda=\lambda_*}$ is nontrivial since otherwise, by (\ref{Differentiating-equation}) $\tilde y(\cdot,\lambda_*)=0$, which is a contradiction. Clearly, ${\partial \tilde y\over \partial \lambda}|_{\lambda=\lambda_*}\in L_W^2([a,b],\mathbb{C}^d)$. If in addition ${\partial \tilde y\over \partial \lambda}|_{\lambda=\lambda_*}$ satisfies the boundary condition $\mathbf{A}$, we have
\begin{align}\label{generalized eigenfunction1}
(T_{(\pmb\omega,\mathbf{A})}-\lambda_*){\partial \tilde y\over \partial \lambda}|_{\lambda=\lambda_*}=\tilde y(\cdot,\lambda_*).
\end{align}
Now we show that ${\partial \tilde y\over \partial \lambda}|_{\lambda=\lambda_*}$ satisfies  $\mathbf{A}$.
Let $\tilde \psi=\sum\limits_{i=1}^pc_{i}\varphi_{i}^{(1)}+\sum\limits_{i=p+1}^{2d}c_{i}\varphi_{i}.$
It follows from (\ref{Contradiction-for-multiplicity-eigenvalue1}) that
\begin{align}\label{contradiction-boundary-tilde-psi-1}
\tilde A \begin{pmatrix}\tilde \psi(a)\\P\tilde \psi'(a)\end{pmatrix}+\tilde B \begin{pmatrix}  \tilde \psi(b)\\P\tilde \psi'(b)\end{pmatrix}=0.
\end{align}
Note that ${\partial y_i\over \partial \lambda}|_{\lambda=\lambda_*}=\varphi_i^{(1)},1\leq i\leq p$, and $y_i(\cdot,\lambda_*)=\varphi_i$,
$p<i\leq 2d$.
Thus
$${\partial \tilde y\over \partial \lambda}|_{\lambda=\lambda_*}=\sum\limits_{i=1}^pc_{i}\varphi^{(1)}_i+\sum\limits_{i=p+1}^{2d}c_{i} \varphi_{i}=\tilde \psi.$$
(\ref{contradiction-boundary-tilde-psi-1}) implies that ${\partial \tilde y\over \partial \lambda}|_{\lambda=\lambda_*}$ satisfies $\mathbf{A}$.
By (\ref{generalized eigenfunction1}), ${\partial \tilde y\over \partial \lambda}|_{\lambda=\lambda_*}$ is a generalized eigenfunction of $\lambda_*$, which contradicts the fact that the geometric and algebraic multiplicities of $\lambda_*$ are equal.

Therefore, $\tilde \Gamma_{(\pmb\omega, \mathbf{A})}(\lambda_*)\neq0$ in any case and $\kappa=p$.
This completes the proof.
\end{proof}
\begin{remark}
Here our proof  is independent of the condition that $P$ is positive.
Thus Theorem \ref{ equivalence of three multiplicities of an eigenvalue} also holds true when $P$ is invertible and  non-positive a.e. on $[a,b]$.
\end{remark}
\section{Continuity of the $n$-th eigenvalue}\label{Continuity of the $n$-th eigenvalue in 1-dimensional case}

In this section, we prove that the $n$-th eigenvalue is continuously dependent  on the Sturm-Liouville equations and  boundary conditions
when restricted into the $k$-th layer, where $0\leq k\leq2d$.

\begin{theorem}\label{omegatimesboundary condition continuity}
 The $n$-th eigenvalue is continuous on $\Omega\times \Sigma_k^\mathbb{C}$ for each $n\geq 1$, where $0\leq k\leq 2d$.
\end{theorem}

\begin{proof} Let $(\pmb\omega_1, \mathbf{A}_1)\in\Omega\times\Sigma_k^\mathbb{C}$, where  $\pmb\omega_1=(P_1,Q_1,W_1)$. By Lemma \ref{continuity principle}, it suffices to show that
there exists a neighborhood $U_1$ of $(\pmb\omega_1, \mathbf{A}_1)$ in $\Omega\times\Sigma_k^\mathbb{C}$ such that the first eigenvalue $\lambda_1$ is bounded from below on $U_1$.
It is equivalent to show that there exists $\mu\in\mathbb{R}$ such that
\begin{align*}
\langle T_{(\pmb\omega,\mathbf{A})}y,y\rangle_W\geq\mu\| y\|^2_W,\;\;\forall \;y\in D_{\mathbf{A}},
\end{align*}
uniformly for $(\pmb\omega,\mathbf{\mathbf{A}})\in U_1$, where $T_{(\pmb\omega,\mathbf{A})}$ and $D_{\mathbf{A}}$ are defined in  (\ref{Sturm-Liouville operator}) and (\ref{domain}).

Firstly, let
$U_1$ be chosen sufficiently small  that for any $(\pmb\omega,\mathbf{A})\in U_1$   with $\pmb\omega=(P, Q, W)$,
there exist  $\delta_1,\mu_1>0$ satisfying  $\|\pmb\omega-\pmb\omega_1\|_{L^\infty\times L^\infty\times L^\infty}<\delta_1$ and
  \begin{align}\label{PtWt}
  P(t)\geq\mu_1,\;\;W(t)\geq\mu_1, \; a.e.\;{\rm on}\; t\in[a,b].
  \end{align}
Direct computations show that for any $(\pmb\omega,\mathbf{A})\in U_1$ with $\pmb\omega=(P,Q,W)$ and any $y\in D_{\mathbf{A}}$,
\begin{align}\label{sequiform}
\langle T_{(\pmb\omega,\mathbf{A})}y,y\rangle_W&=\int_a^b(-(Py')'+Qy,y)_ddt\\\nonumber
&=\int_a^b(Py',y')_d+(Qy,y)_ddt-\left(\begin{pmatrix}Py'(a)\\Py'(b)\end{pmatrix},\begin{pmatrix}-y(a)\\y(b)\end{pmatrix}\right)_{2d},
\end{align}
where $(\cdot,\cdot)_d$ is the usual inner product in $\mathbb{C}^d$.

Let $\mathbf{A}=[A\;|\;B]\in \Sigma_k^\mathbb{C}$.
Choose  $T_0\in GL(2d,\mathbb{C})$ such that
$$T_0(A\;|\;B)=\begin{pmatrix}A_1&B_1\\A_2&B_2\end{pmatrix},$$
where $A_1,B_1\in\mathcal{M}_{(2d-k)\times 2d}$, $A_2,B_2\in\mathcal{M}_{k\times 2d}$ and $\mathrm{rank}B_1=\mathrm{rank}B=2d-k$.
Then exists the unique $E\in \mathcal{M}_{k\times(2d-k)}$ such that $B_2=EB_1$, and $E$ is locally  continuously dependent  on $B$. Direct
computation shows that
$$\begin{pmatrix}I_{2d-k}&0\\-E&I_k\end{pmatrix}\begin{pmatrix}A_1&B_1\\A_2&B_2\end{pmatrix}=\begin{pmatrix}A_1&B_1\\A_2-EA_1&0\end{pmatrix}.$$
Applying the $QR$  decomposition (see Theorem 2.1.14 in \cite{Horn2013}) on $B_1$ and $A_2-EA_1$, we get
$$B_1=L_1D_1,\ A_2-EA_1=L_2C_2,$$
where $L_1\in GL(2d-k,\mathbb{C})$ and  $L_2\in GL(k,\mathbb{C})$ are lower triangular  positive matrices, and $D_1\in\mathcal{M}_{(2d-k)\times 2d}$ and $C_2\in\mathcal{M}_{k\times 2d}$ satisfy $D_1D_1^*=I_{2d-k}$, $C_2C_2^*=I_k$. Let $C_1=L_1^{-1}A_1$. Note that $L_i,  C_i$ and $D_1$, $i=1,2$,  are uniquely determined and locally continuously dependent on $\mathbf{A}$. Thus we denote them by $L_i(\mathbf{A}),  C_i(\mathbf{A})$ and $D_1(\mathbf{A})$.  Direct calculation gives
$$\begin{pmatrix}L(\mathbf{A})^{-1}_1&0\\0&L(\mathbf{A})^{-1}_2\end{pmatrix}\begin{pmatrix}A_1&B_1\\A_2-EA_1&0\end{pmatrix}
=\begin{pmatrix}C_1(\mathbf{A})&D_1(\mathbf{A})\\C_2(\mathbf{A})&0\end{pmatrix}.$$
Since
$\begin{pmatrix}C_1(\mathbf{A})\\C_2(\mathbf{A})\end{pmatrix}\begin{pmatrix}D_1(\mathbf{A})^*&0\end{pmatrix}=
\begin{pmatrix}D_1(\mathbf{A})\\0\end{pmatrix}\begin{pmatrix}C_1(\mathbf{A})^*&C_2(\mathbf{A})^*\end{pmatrix}$,
we have
$$C_2(\mathbf{A})D_1(\mathbf{A})^*=0\ {\rm and}\ C_1(\mathbf{A})D_1(\mathbf{A})^*=D_1(\mathbf{A})C_1(\mathbf{A})^*.$$
Thus $E_1(\mathbf{A})E_1(\mathbf{A})^*=I_{2d}$,
which implies that $E_1(\mathbf{A})$ is unitary, where $E_1(\mathbf{A}):=\begin{pmatrix}C_2(\mathbf{A})\\D_1(\mathbf{A})\end{pmatrix}$.
It follows from the boundary condition  that
$$C_1(\mathbf{A})\begin{pmatrix}-y(a)\\y(b)\end{pmatrix}+D_1(\mathbf{A})\begin{pmatrix}Py'(a)\\Py'(b)\end{pmatrix}=0,\ C_2(\mathbf{A})\begin{pmatrix}-y(a)\\y(b)\end{pmatrix}=0.$$
Then
\begin{align}\label{second term estimate}
&\left(\begin{pmatrix}Py'(a)\\Py'(b)\end{pmatrix},\begin{pmatrix}-y(a)\\y(b)\end{pmatrix}\right)_{2d}\\
\nonumber
=&\left(E_1(\mathbf{A})\begin{pmatrix}Py'(a)\\Py'(b)\end{pmatrix},E_1(\mathbf{A})\begin{pmatrix}-y(a)\\y(b)\end{pmatrix}\right)_{2d}\\\nonumber
=&\left(D_1(\mathbf{A})\begin{pmatrix}Py'(a)\\Py'(b)\end{pmatrix},D_1(\mathbf{A})\begin{pmatrix}-y(a)\\y(b)\end{pmatrix}\right)_{2d}\\\nonumber
=&\left(-C_1(\mathbf{A})\begin{pmatrix}-y(a)\\y(b)\end{pmatrix},D_1(\mathbf{A})\begin{pmatrix}-y(a)\\y(b)\end{pmatrix}\right)_{2d}\\\nonumber
\leq& c\|C_1(\mathbf{A})\|_{\mathcal{M}_{(2d-k)\times 2d}}\|D_1(\mathbf{A})\|_{\mathcal{M}_{(2d-k)\times 2d}}(|y(a)|_{d}^2+|y(b)|_{d}^2)\\\nonumber
\leq& c\|C_1(\mathbf{A})\|_{\mathcal{M}_{(2d-k)\times 2d}}\|D_1(\mathbf{A})\|_{\mathcal{M}_{(2d-k)\times 2d}}\|y\|_{C^0}^2,
\end{align}
where $\|C\|_{\mathcal{M}_{(2d-k)\times 2d}}=\max_{ij}\{|c_{ij}|\}$ and $\|y\|_{C^0}=\max_{t\in [a,b]}\{|y(t)|_{d}\}$.
 Here and in the sequel, $c$ denotes a generic positive constant and $c(\alpha)$ denotes such a constant depending only on $\alpha$. We now interpolate $\|y\|_{C^0}$ between the norms $\|y\|_{L^2}$ and $\|y'\|_{L^2}$. Choose $\varepsilon_1>0$ sufficiently small and $c(\varepsilon_1)>0$  sufficiently large  such that
\begin{align}\label{varepsilon chosen}
c\|C_1(\mathbf{A}_1)\|_{\mathcal{M}_{(2d-k)\times 2d}}\|D_1(\mathbf{A}_1)\|_{\mathcal{M}_{(2d-k)\times 2d}}\varepsilon_1<\mu_1
\end{align}
and
\begin{align}\label{sobolev interpolate}
 \|y\|^2_{C^0}\leq \varepsilon_1\|y'\|^2_{L^2}+c(\varepsilon_1)\|y\|^2_{L^2}.
 \end{align}
 It follows from (\ref{varepsilon chosen}) and the locally continuity of $C_1(\mathbf{A})$ and $D_1(\mathbf{A})$ on $\mathbf{A}$   that
 $U_1$ can be shrunk such that for any $(\pmb\omega,\mathbf{A})\in U_1$,
 \begin{align}\label{varepsilon chosen1}
c\|C_1(\mathbf{A})\|_{\mathcal{M}_{(2d-k)\times 2d}}\|D_1(\mathbf{A})\|_{\mathcal{M}_{(2d-k)\times 2d}}\varepsilon_1<\mu_1.
\end{align}
By (\ref{PtWt})--(\ref{second term estimate}), (\ref{sobolev interpolate})--(\ref{varepsilon chosen1}) and noting that $\|\cdot\|_{L^2}$ is equivalent to $\|\cdot\|_{L_W^2}$, we have
\begin{align*}
&\langle T_{(\pmb\omega,\mathbf{A})}y,y\rangle_W\\
\geq&\mu_1\|y'\|^2_{L^2}-(\|Q_1\|_{L^\infty}+\delta_1)\|y\|^2_{L^2}
-c\|C_1(\mathbf{A})\|_{\mathcal{M}_{(2d-k)\times 2d}}\|D_1(\mathbf{A})\|_{\mathcal{M}_{(2d-k)\times 2d}}\cdot\\
&\varepsilon_1\|y'\|^2_{L^2}-c\|C_1(\mathbf{A})\|_{\mathcal{M}_{(2d-k)\times 2d}}\|D_1(\mathbf{A})\|_{\mathcal{M}_{(2d-k)\times 2d}}c(\varepsilon_1)\|y\|^2_{L^2}\\
\geq&\mu\|y\|^2_{L_W^2}
\end{align*}
for some $\mu\in\mathbb{R}$ and for all $(\pmb\omega,\mathbf{A})\in U_1$. The proof is complete.
\end{proof}

\section{Singularity of the $n$-th eigenvalue}\label{Singularity of the $n$-th eigenvalue of  high dimensional Sturm-Liouville problems}

In this section, we determine the singular set in the space of boundary conditions for  $d$-dimensional Sturm-Liouville problems, and  give complete characterization of  asymptotic behavior of the $n$-th eigenvalue  near any fixed singular  boundary condition.

\begin{theorem}\label{discontinuity-main-theorem}
 Consider the restriction of $\lambda_n$ to $\mathcal{O}_K^{\mathbb{C}}$ for each $n\geq 1$, where $K\subset\{1,\cdots,2d\}$.
\begin{itemize}\label{main theorem}
\item[{\rm (i)}] The restriction of $\lambda_n$ to $\mathcal{O}_\emptyset^{\mathbb{C}}$ is continuous.
\item[{\rm (ii)}] Let $K$ be  nonempty, $0\leq n^0<n^0_0\leq \sharp(K)$, $n^+\geq n_0^+$ and $n^-\geq n_0^-$. Then
for any $\mathbf{A}\in J^{(n_0^0,n_0^+,n_0^-)}_{\mathcal{O}_{K}^{\mathbb{C}}}$, we have
\begin{align*}
\lim\limits_{J^{(n^0,n^+,n^-)}_{\mathcal{O}_{K}^{\mathbb{C}}}\ni\mathbf{B}\to\mathbf{A}}\lambda_n(\mathbf{B})&=-\infty,\;n\leq n^+-n_0^+,\\
\lim\limits_{J^{(n^0,n^+,n^-)}_{\mathcal{O}_{K}^{\mathbb{C}}}\ni\mathbf{B}\to\mathbf{A}}\lambda_n(\mathbf{B})&=
\lambda_{n-(n^+-n_0^+)}(\mathbf{A}),\;n >n^+-n_0^+,
\end{align*}
where $J^{(n^0,n^+,n^-)}_{\mathcal{O}_{K}^{\mathbb{C}}}$ is defined in (\ref{def-Jn-n0n+}).
\end{itemize}
 Consequently, $\Sigma^\mathbb{C}=\{\mathbf{A}=[A\;|\;B]:n^0(B)>0\}.$
\end{theorem}

\begin{proof} (i) holds due to Theorem \ref{omegatimesboundary condition continuity} and the fact $\mathcal{O}_{\emptyset}^\mathbb{C}\subset \Sigma_0^\mathbb{C}$.  We shall give the proof of Theorem \ref{main theorem} (ii) step by step  via  the following results in Lemmas \ref{one path from 1-dim}, \ref{asymptotic behavoir for tilde A0}, \ref{asymptotic behavior A0} and \ref{asymptotic behavior C0 any sturm liouville equation}.
\end{proof}

By Theorems \ref{omegatimesboundary condition continuity}  and  \ref{discontinuity-main-theorem}, we give the explicit description of $U^i$ in (\ref{U-i1}) and (\ref{U-i2}):
\begin{corollary}\label{corollary-u-i}
For each $\mathbf{A}\in J^{(n_0^0,n_0^+,n_0^-)}_{\mathcal{O}_{K}^{\mathbb{C}}}$ and $0\leq i\leq n_0^0$, there exists a neighborhood $U$ of $\mathbf{A}$ in $\mathcal{O}_{K}^{\mathbb{C}}$ such that
$$U^i=\bigcup_{n^+-n_0^+=i,n^0\leq n^0_0,n^-\geq n^-_0}J^{(n^0,n^+,n^-)}_{\mathcal{O}_{K}^{\mathbb{C}}}\cap U,$$
where $U^i$ satisfies (\ref{U-i1}) and (\ref{U-i2}).
\end{corollary}

Let $K\subset\{1,\cdots,2d\}$  be a nonempty subset below. For each $0\leq n^0\leq \sharp(K)$,
\begin{align*}
\mathcal{O}_K^{\mathbb{C}}\cap\Sigma_{n^0}^\mathbb{C}&\begin{cases} \neq\emptyset & {\rm if}\;\;  0\leq n^0\leq \sharp(K), \\
                      =\emptyset &  {\rm if}\;\; \sharp(K)< n^0\leq 2d, \end{cases}
\end{align*}
and furthermore, $n^0(S_K(\mathbf{A}))=n^0$ for any $\mathbf{A}\in\mathcal{O}_K^{\mathbb{C}}\cap\Sigma_{n^0}^\mathbb{C}$.  $\mathcal{O}_K^{\mathbb{C}}\cap\Sigma_{n^0}^\mathbb{C}$ possesses precisely $\sharp(K)-n^0+1$ components as follows:
\begin{align}\label{path connected components}
J^{(n^0,n^+,n^-)}_{\mathcal{O}_{K}^{\mathbb{C}}},\;0\leq n^+\leq \sharp(K)-n^0,\;n^-=\sharp(K)-n^0-n^+.
\end{align}

Next, we show that every $J^{(n^0,n^+,n^-)}_{\mathcal{O}_{K}^{\mathbb{C}}}$ is a path connected component.
In the following discussion, by a path $\gamma$ to connect $x_0$ and $x_1$  in a topological space $X$, we mean a continuous function $\gamma:[0,1]\to X$ such that $\gamma(0)=x_0$ and $\gamma(1)=x_1$.

\begin{lemma}\label{JC path connected}
 $J^{(n^0,n^+,n^-)}_{\mathcal{O}_{K}^{\mathbb{C}}}$ is path connected for each $0\leq n^0\leq \sharp(K)$ and   $0\leq n^+\leq \sharp(K)-n^0$.
\end{lemma}

\begin{proof}
Fix any given  $\mathbf{A}_j\in J^{(n^0,n^+,n^-)}_{\mathcal{O}_{K}^{\mathbb{C}}}, j=1,2$, and denote
\begin{align*}
\mathbf{A}_j=[A_j\;|\;B_j]=&[(a^{(j)}_1,a^{(j)}_2,\cdots,a^{(j)}_{2d})\;|\;(b^{(j)}_1,b^{(j)}_2,\cdots,b^{(j)}_{2d})],
\end{align*}
where
\begin{align*}
a_{i}^{(j)}=&\begin{cases} -e_i & {\rm if}\;\;  i\in K, \\
                      s_i^{(j)} &  {\rm if}\;\; i\in\{1,2,\cdots,2d\}\backslash K, \end{cases}\\\nonumber
 b_{i}^{(j)}=&\begin{cases}  s_i^{(j)} &{\rm if}\;\;  i\in K, \\
                      e_i &  {\rm if}\;\; i\in\{1,2,\cdots,2d\}\backslash K,\end{cases}
\end{align*}
 for two different $2d\times2d$ Hermitian matrices  $S^{(j)}=(s^{(j)}_1,\cdots, s^{(j)}_{2d})$, $s_i^{(j)}=(s_{1i}^{(j)},\cdots, s_{2di}^{(j)})^T$. Let  $m_0:=\sharp(K)$ and
 $S_K(\mathbf{A}_j)$ be defined as that in (\ref{def-SK}). Then there exist  matrices $R^{(j)}\in GL(m_0,\mathbb{C})$, $j=1,2,$ such that
 \begin{align}\label{Ajtransform}
 S_K(\mathbf{A}_j)=R^{(j)*}\hat JR^{(j)}, \hat J:=\begin{pmatrix}0_{n^0}&&\\
&I_{n^+}&\\
&&-I_{n^-}\end{pmatrix}.
 \end{align}
 Choose a path  of $m_0\times m_0$ matrices: $\gamma$ to connect  $R^{(1)}$ and $R^{(2)}$  such that $\gamma(\tau)\in GL(m_0,\mathbb{C})$, $\tau\in[0,1]$.
 Then $\gamma_0(\tau):= \gamma(\tau)^* \hat J \gamma(\tau)$
with entries $(\gamma_0(\tau))_{ij}$, $1\leq i,j\leq m_0$,  is a path to connect $S_K(\mathbf{A}_1)$ and $S_K(\mathbf{A}_2)$ such that $n^0(\gamma_0(\tau))=n^0$ and $n^{\pm}(\gamma_0(\tau))=n^{\pm}$, $\tau\in[0,1]$.
Define a Hermitian  $2d\times 2d$ matrix $S^{(\tau)}=(s^{(\tau)}_1,\cdots,s^{(\tau)}_{2d})$, $\tau\in[0,1]$, where $s_i^{(\tau)}=(s_{1i}^{(\tau)},\cdots,s_{2di}^{(\tau)})^T$ with entries
\begin{align*}
s_{li}^{(\tau)}:=\begin{cases}  (1-\tau)s_{li}^{(1)}+\tau s_{li}^{(2)} & \;\;\;{\rm if}\;\;  i\notin K\; {\rm or} \;l\notin K, \\
                      (\gamma_0(\tau))_{li} &  \;\;\;{\rm if}\;\; l,i\in K. \end{cases}
\end{align*}
 Thus we can construct a path $\xi$ in $J^{(n^0,n^+,n^-)}_{\mathcal{O}_{K}^{\mathbb{C}}}$ to connect  $\mathbf{A}_1$ and $\mathbf{A}_2$:
\begin{align*}
\xi{(\tau)}=[A^{(\tau)}\;|\;B^{(\tau)}]
=&[(a^{(\tau)}_1,a^{(\tau)}_2,\cdots,a^{(\tau)}_{2d})\;|\;(b^{(\tau)}_1,b^{(\tau)}_2,\cdots,b^{(\tau)}_{2d})],
\end{align*}
where
\begin{align*}
a_{i}^{(\tau)}=&\begin{cases} -e_i & {\rm if}\;\;  i\in K, \\
                      s_i^{(\tau)} &  {\rm if}\;\; i\in\{1,2,\cdots,2d\}\backslash K, \end{cases}\\\nonumber
 b_{i}^{(\tau)}=&\begin{cases}  s_i^{(\tau)} & {\rm if}\;\;  i\in K, \\
                      e_i & {\rm if}\;\; i\in\{1,2,\cdots,2d\}\backslash K.\end{cases}
\end{align*}
 This finishes the proof.
\end{proof}

\begin{remark}
The only difference in the proof of Lemma \ref{JC path connected}  for real boundary conditions is that
$R^{(j)}$ should be chosen such that $\det R^{(j)}>0$, $j=1,2$. This can be easily done, since otherwise, we can replace  $R^{(j)}$  by
  \begin{align*}
\begin{pmatrix}-1&\\
&I_{m_0-1}\end{pmatrix}R^{(j)}.
 \end{align*}
\end{remark}

\begin{lemma}\label{Uvarepsilon conponent}
Let  $0\leq n_0^0\leq \sharp(K)$ and $0\leq n_0^+ \leq \sharp(K)-n_0^0$.  Then for any ${\mathbf{A}}\in J^{(n_0^0,n_0^+,n_0^-)}_{\mathcal{O}_{K}^{\mathbb{C}}}$, there exists $\varepsilon_1>0$ such that for any $0<\varepsilon<\varepsilon_1$,
\begin{align}\label{Uvarepsilon}\nonumber
U_\varepsilon:=&\{\mathbf{B}\in\mathcal{O}_{K}^{\mathbb{C}}:\|S(\mathbf{B})-S(\mathbf{A})\|_{\mathcal{M}_{2d\times 2d}}<\varepsilon\}\\
=&\bigcup_{n^0\leq n_0^0,n^+\geq n_0^+,n^-\geq n_0^-}U_\varepsilon\cap J^{(n^0,n^+,n^-)}_{\mathcal{O}_{K}^{\mathbb{C}}},
\end{align}
and $U_\varepsilon\cap J^{(n^0,n^+,n^-)}_{\mathcal{O}_{K}^{\mathbb{C}}}$ is path connected for any $n^0\leq n_0^0$, $n^+\geq n_0^+$ and $n^-\geq n_0^-$.
\end{lemma}

\begin{proof} Let ${\mathbf{A}}=[A\;|\;B]$ be given in (\ref{coordinate}). Then there exists a $m_0\times m_0$ unitary matrix $N$ such that
 \begin{align}\label{Atransform}
S_K(\mathbf{A})=N^*\begin{pmatrix}0_{n_0^0}&\\
&M\end{pmatrix}N,\;\;&M=\begin{pmatrix}
\mu_1&&&&&\\&\ddots&&&&\\&&\mu_{n_0^+}&&&\\
&&&\nu_1&&\\&&&&\ddots&\\&&&&&\nu_{n_0^-}\end{pmatrix},
 \end{align}
 where  $m_0=\sharp(K)$ and $\nu_1\leq\cdots\leq\nu_{n_0^-}<0<\mu_1\leq\cdots\leq\mu_{n_0^+}$.
(\ref{Uvarepsilon}) is straightforward  from the small perturbation of $S_K(\mathbf{A})$ in  (\ref{Atransform}).
Fix any $\mathbf{B}_i\in U_\varepsilon\cap J^{(n^0,n^+,n^-)}_{\mathcal{O}_{K}^{\mathbb{C}}}, i=1,2.$ Their entries are given by a similar way as  (\ref{coordinate}). Then
the connection from $s_{lj}(\mathbf{B}_1)$ to $s_{lj}(\mathbf{B}_2)$ is trivial if  $l\notin K$ or $j\notin K$. So it suffices to construct a path connecting $S_K(\mathbf{B}_1)$ and $S_K(\mathbf{B}_2)$. To do so, we only need to show that
\begin{align*}
\mathcal{E}:=\{E: \|E-E_0\|_{\mathcal{M}_{m_0\times m_0}}<\varepsilon_0,E=E^*,n^\pm(E)=n^\pm\}
\end{align*}
is path connected for $\varepsilon_0>0$ sufficiently small, where $E_0:=\begin{pmatrix}0_{n_0^0}&\\
&M\end{pmatrix}$.
Let $E=\begin{pmatrix}E_{11}&E^*_{12}\\
E_{12}&M+E_{22}\end{pmatrix}\in\mathcal{E}$ and define $F_{\tau}=\begin{pmatrix}1&0\\
-\tau(M+E_{22})^{-1}E_{12}&1\end{pmatrix}$, where $E_{ii}=E_{ii}^*$, $i=1,2$, and $\tau\in[0,1]$.
Then $\gamma(\tau)=F^*_{\tau}EF_{\tau}$ is a path in $\mathcal{E}$ from $E$ to
$\begin{pmatrix}F&0\\
0&M+E_{22}\end{pmatrix}$, where $F=E_{11}-E_{12}^*(M+E_{22})^{-1}E_{12}$. Similarly, for another $\tilde E=\begin{pmatrix}\tilde E_{11}&\tilde E^*_{12}\\
\tilde E_{12}&M+\tilde E_{22}\end{pmatrix}\in \mathcal{E}$,
one can construct a path $\tilde\gamma$ in  $\mathcal{E}$ to connect $\tilde E$ to
$\begin{pmatrix}\tilde F&0\\
0&M+\tilde E_{22}\end{pmatrix}$. Now, we  connect $\begin{pmatrix}F&0\\
0&M+E_{22}\end{pmatrix}$ and $\begin{pmatrix}\tilde F&0\\
0&M+\tilde E_{22}\end{pmatrix}$ in $\mathcal{E}$.   $M+(1-\tau)E_{22}+\tau\tilde E_{22}$ is a path from $M+E_{22}$ to $M+\tilde E_{22}$. The rest is to connect $F$ and $\tilde F$.
Note that  $n^{\pm}(F)=n^{\pm}(\tilde F)=n^{\pm}-n_0^{\pm}$. Thus there exists a path $ \gamma_1$ to connect
 $F$ and $\tilde F$  by a similar strategy used in (\ref{Ajtransform}) such that  $n^{\pm}(\gamma_1(\tau))=n^{\pm}-n_0^{\pm}$ for each $\tau\in[0,1]$. However, $\|\gamma_1(\tau)\|$ maybe larger than $\varepsilon_0$, where   $\|\cdot\|:=\|\cdot\|_{\mathcal{M}_{n_0^0\times n_0^0}}$ for convenience. Thus we need to shrink the path $\gamma_1$ as follows. Assume that $\|\tilde F\|>\|F\|$, otherwise the construction is similar.  Connect $F$ and  ${\|F\|\over \|\tilde F\|}\tilde F$ by ${\|F\|\over \|\gamma_1(\tau)\|}\gamma_1(\tau)$,  and then connect  ${\|F\|\over \|\tilde F\|}\tilde F$ and $\tilde F$ by ${(1-\tau)\|F\|+\tau \|\tilde F\|\over \|\tilde F\|}\tilde F$, $\tau\in[0,1]$.  This completes  the proof.
\end{proof}

{\it In the following discussion, we always assume that  $0< n_0^0\leq \sharp(K)$, $0\leq n_0^+ \leq \sharp(K)-n_0^0$,   $n^0<n_0^0$, $n^+\geq n^+_0$ and $n^-\geq n^-_0$.}\medskip

 Set $K=\{k_1,\cdots,k_{m_0}\}$ and define
\begin{align}\label{independent boundary condition}
\tilde{\mathbf{A}}_0:=[A_{(n_0^0,n_0^+)}\;|\;B_{(n_0^0,n_0^+)}]=&[-I_{2d}\;|\;(b_1,b_2,\cdots,b_{2d})]\in J^{(n_0^0,n_0^+,n_0^-)}_{\mathcal{O}_{K}^{\mathbb{C}}},
\end{align}
where
\begin{align*}
 b_{l}=&\begin{cases}  0 & \;\;\;{\rm if}\;\; l \in \{k_1,\cdots,k_{n_0^0}\}, \\
                      e_l &  \;\;\;{\rm if}\;\; l\in\{k_{n_0^0+1},\cdots,k_{n_0^0+n_0^+}\}\cup(\{1,\cdots,2d\}\backslash K),\\
                      -e_l &  \;\;\;{\rm if}\;\; l\in\{k_{n_0^0+n_0^++1},\cdots,k_{m_0}\}.\end{cases}
\end{align*}

Consider
\begin{align}\label{independent equation}
P_0:=\begin{pmatrix}p_{11}&&\\
&\ddots&\\
&&p_{dd}\end{pmatrix},
\end{align}
 $Q_0$ and $W_0$ are defined similarly as (\ref{independent equation}), where $p_{ii}, q_{ii}, w_{ii}\in L^{\infty}([a,b],\mathbb{R})$, and $p_{ii}, w_{ii}>0$ a.e. on $[a,b]$   for each $1\leq i\leq d$.
We get the following result from $1$-dimensional Sturm-Liouville problems:

\begin{lemma}\label{one path from 1-dim}
Consider the Sturm-Liouville equation  $\pmb\omega_0=(P_0,Q_0,W_0)$ defined in (\ref{independent equation}).  Then for $\tilde{\mathbf{A}}_0\in J^{(n_0^0,n_0^+,n_0^-)}_{\mathcal{O}_{K}^{\mathbb{C}}}$ defined in (\ref{independent boundary condition}), there exists a path $\tilde{\mathbf{A}}_s\in J^{(n^0,n^+,n^-)}_{\mathcal{O}_{K}^{\mathbb{C}}}$, $s\in(0,1]$ such that $\tilde{\mathbf{A}}_s\to\tilde{\mathbf{A}}_0$ as $s\to0^+$, and
\begin{align}\label{d one-dimensional asymptotic behavior1}
\lim\limits_{s\to0^+}\lambda_n(\tilde{\mathbf{A}}_s)&=-\infty,\;n\leq n^+-n_0^+,\\\label{d one-dimensional asymptotic behavior2}
\lim\limits_{s\to0^+}\lambda_n(\tilde{\mathbf{A}}_s)&=
\lambda_{n-(n^+-n_0^+)}(\tilde{\mathbf{A}}_0),\;n >n^+-n_0^+.
\end{align}
\end{lemma}
\begin{proof}
$\tilde{\mathbf{A}}_s\in J^{(n^0,n^+,n^-)}_{\mathcal{O}_{K}^{\mathbb{C}}}$, $s\in(0,1]$, can be directly constructed by setting
\begin{align*}
\tilde{\mathbf{A}}_s=&[-I_{2d}\;|\;(b_{1}(s),b_{2}(s),\cdots,b_{2d}(s))]\in J^{(n^0,n^+,n^-)}_{\mathcal{O}_{K}^{\mathbb{C}}},
\end{align*}
where
\begin{align*}
 b_{l}(s)=&\begin{cases}  0 & \;\;\;{\rm if}\;\; l \in \{k_1,\cdots,k_{n^0}\}, \\
                      s e_l& \;\;\;{\rm if}\;\;l\in\{k_{n^0+1},\cdots,k_{n^0+n^+-n_0^+}\},\\
                      -s e_l& \;\;\;{\rm if}\;\;l\in\{k_{n^0+n^+-n_0^++1},\cdots,k_{n_0^0}\},\\
                      e_l &  \;\;\;{\rm if}\;\; l\in\{k_{n_0^0+1},\cdots,k_{n_0^0+n_0^+}\}\cup(\{1,\cdots,2d\}\backslash K),\\
                      -e_l &  \;\;\;{\rm if}\;\; l\in\{k_{n_0^0+n_0^++1},\cdots,k_{m_0}\},\end{cases}
\end{align*}
and $b_{l}(s):=(b_{1l}(s),\cdots,b_{2dl}(s))^T$. Since the Sturm-Liouville equation $\pmb\omega_0$ and  $\tilde{\mathbf{A}}_s$ are equivalent to $d$ one-dimensional Sturm-Liouville equations
\begin{align*}
-(p_{jj}y_j')'+q_{jj}y_j=\lambda w_{jj}y_j,\;\;{\rm on}\;\;[a,b],
\end{align*}
with boundary conditions
\begin{align*}
y_j(a)+b_{jj}(s)(p_{jj}y'_j)(a)=0,\;\;-y_j(b)+b_{d+j\;d+j}(s)(p_{jj}y'_j)(b)=0,
\end{align*}
where $1\leq j\leq d$.
Then  (\ref{d one-dimensional asymptotic behavior1}) and (\ref{d one-dimensional asymptotic behavior2}) hold by Proposition \ref{1dimrefinement}.
\end{proof}

Next, we consider other paths  in  $J^{(n_0^0,n_0^+,n_0^-)}_{\mathcal{O}_{K}^{\mathbb{C}}}$, which tend to $\tilde{\mathbf{A}}_0$ given in (\ref{independent boundary condition}).

\begin{lemma}\label{asymptotic behavoir for tilde A0}
Let  $\pmb\omega_0=(P_0,Q_0,W_0)$ be given in (\ref{independent equation}).  Then for $\tilde{\mathbf{A}}_0\in J^{(n_0^0,n_0^+,n_0^-)}_{\mathcal{O}_{K}^{\mathbb{C}}}$ defined in (\ref{independent boundary condition}) and for any path $\mathbf{A}_s\in J^{(n^0,n^+,n^-)}_{\mathcal{O}_{K}^{\mathbb{C}}}$, $s\in(0,1]$ such that $\mathbf{A}_s\to\tilde{\mathbf{A}}_0$ as $s\to0^+$, we have
\begin{align}\label{any path d one-dimensional asymptotic behavior1}
\lim\limits_{s\to0^+}\lambda_n(\mathbf{A}_s)&=-\infty,\;n\leq n^+-n_0^+,\\\label{any path d one-dimensional asymptotic behavior2}
\lim\limits_{s\to0^+}\lambda_n(\mathbf{A}_s)&=
\lambda_{n-(n^+-n_0^+)}(\tilde{\mathbf{A}}_0),\;n >n^+-n_0^+.
\end{align}
\end{lemma}
\begin{proof}
If $n^+-n_0^+=0$, then it suffices to show that $\lambda_1$ is bounded from below on $\{\mathbf{A}_s:s\in(0,1]\}$ by Lemma \ref{continuity principle}.
Suppose otherwise, there exists a  sequence $\{s_n\}_{n=1}^\infty$ such that $s_n\to 0^+$ as $n\to\infty$ and
\begin{align}\label{to be contradict}
\lim\limits_{n\to\infty}\lambda_1({\mathbf{A}}_{s_n})=-\infty.
\end{align}
Let $r_1< r_2$  such that $\lambda_1(\tilde{\mathbf{A}}_0)$ with multiplicity $m_0$ is the only eigenvalue of $(\pmb\omega_0,\tilde{\mathbf{A}}_0)$ in $(r_1,r_2)$ and neither $r_1$ nor $r_2$ is an eigenvalue of $(\pmb\omega_0,\tilde{\mathbf{A}}_0)$.
By Lemma  \ref{local continuity of eigenvalues}, there exists a neighborhood $U_0$  of $\tilde{\mathbf{A}}_0$  in $\mathcal{O}_K^{\mathbb{C}}$ such that for each $\mathbf{A}\in U_0$, $(\pmb\omega_0,\mathbf{A})$ has exactly $m_0$ eigenvalues in $(r_1,r_2)$ and neither $r_1$ nor $r_2$ is an eigenvalue of $(\pmb\omega_0,\mathbf{A})$.
Choose $n_1\in\mathbb{N}$ such that
 \begin{align}\label{one-contradiction by otherwise}
 {\mathbf{A}}_{s_{n_1}}\in U_0\; {\rm and}\; \lambda_1({\mathbf{A}}_{s_{n_1}})<r_1
\end{align}
  by (\ref{to be contradict}). Lemma \ref{one path from 1-dim} tells us there exists $s'\in(0,1]$ such that
 \begin{align}\label{one-contradiction by one path from 1-dim}
 \tilde{\mathbf{A}}_{s'}\in U_0,\;{\rm and}\;\lambda_1(\tilde{\mathbf{A}}_{s'})>r_1.
 \end{align}
 Note that ${\mathbf{A}}_{s_{n_1}}, \tilde{\mathbf{A}}_{s'}\in U_0\cap J^{(n^0,n^+,n^-)}_{\mathcal{O}_{K}^{\mathbb{C}}}$, which can be chosen such that it is path connected by Lemma \ref{Uvarepsilon conponent}. Then we choose a path $\gamma_0$ in $U_0\cap J^{(n^0,n^+,n^-)}_{\mathcal{O}_{K}^{\mathbb{C}}}$ to connect ${\mathbf{A}}_{s_{n_1}}$ and $\tilde{\mathbf{A}}_{s'}$.
By Theorem \ref{omegatimesboundary condition continuity}, $\lambda_1$ is continuous on $\gamma_0$.
Then there exists $\tau_0\in(0,1)$ such that $\lambda_1(\gamma_0({\tau_0}))=r_1$ by (\ref{one-contradiction by otherwise})--(\ref{one-contradiction by one path from 1-dim}).
However, $r_1$   is not an eigenvalue of $(\pmb\omega_0,\mathbf{A})$ for any  $\mathbf{A}\in U_0$. This is a contradiction.

Let $n^+-n_0^+>0$. For any $1\leq i\leq n^+-n_0^+$, assume that
\begin{align}\label{n=i}
\lim\limits_{s\to0^+}\lambda_j(\mathbf{A}_s)=-\infty, \;\;1\leq j\leq i-1,
\end{align}
we show that
\begin{align}\label{n=i+1}
\lim\limits_{s\to0^+}\lambda_{i}(\mathbf{A}_s)=-\infty.
\end{align}
Suppose otherwise, there exists a sequence $\{ s^{(i)}_n\}_{n=1}^\infty\subset (0,1]$ such that $s_n^{(i)}\to 0^+$ as $n\to\infty$, and $\lambda_i$ is bounded from below on $\{\mathbf{A}_{ s^{(i)}_n}\}_{n=1}^\infty$. Then
$\lim\limits_{n\to\infty}\lambda_i(\mathbf{A}_{ s^{(i)}_n})$ $=\lambda_1(\tilde{\mathbf{A}}_0)$ by Lemma \ref{continuity principle} for $i=1$ and  by Lemma \ref{continuity-change-indices} for $i>1$.
Again from Lemma \ref{one path from 1-dim}, $\lim\limits_{s\to0^+}\lambda_i(\tilde{\mathbf{A}}_s)=-\infty$.
Thus we can choose  $\tilde{\mathbf{A}}_{ s'_1},\mathbf{A}_{s^{(i)}_{n_1}}\in U_0\cap J^{(n^0,n^+,n^-)}_{\mathcal{O}_{K}^{\mathbb{C}}}$ such that $\lambda_i(\tilde{\mathbf{A}}_{s'_1})<r_1<\lambda_i(\mathbf{A}_{ s^{(i)}_{n_1}})$ for
some
$s'_1\in(0,1]$ and  some  $n_1\in\mathbb{N}$. Then the choice of $r_1$ contradicts that
$\lambda_i(U_0\cap J^{(n^0,n^+,n^-)}_{\mathcal{O}_{K}^{\mathbb{C}}}):=\{\lambda_i(\mathbf{A}):\mathbf{A}\in U_0\cap J^{(n^0,n^+,n^-)}_{\mathcal{O}_{K}^{\mathbb{C}}}\}$ is connected by Theorem \ref{omegatimesboundary condition continuity} and Lemma \ref{Uvarepsilon conponent}.

To prove (\ref{any path d one-dimensional asymptotic behavior2}), it is sufficient to show that $\lambda_{n^+-n_0^++1}$ is bounded from below
on $U_0\cap J^{(n^0,n^+,n^-)}_{\mathcal{O}_{K}^{\mathbb{C}}}$ by Lemma \ref{continuity-change-indices}. Suppose otherwise, there exists $s_2'\in(0,1]$ such that  $\mathbf{A}_{s_2'}\in U_0\cap J^{(n^0,n^+,n^-)}_{\mathcal{O}_{K}^{\mathbb{C}}}$ and $\lambda_{n^+-n_0^++1}(\mathbf{A}_{s_2'})<r_1$. By Lemma \ref{one path from 1-dim}, there exists $s_3'\in(0,1]$ such that $\tilde{\mathbf{A}}_{s_3'}\in  U_0\cap J^{(n^0,n^+,n^-)}_{\mathcal{O}_{K}^{\mathbb{C}}}$  and $\lambda_{n^+-n_0^++1}(\tilde{\mathbf{A}}_{s_3'})>r_1$. Connect $\mathbf{A}_{s_2'}$ and $\tilde{\mathbf{A}}_{s_3'}$ by a path $\gamma_1$ in $U_0\cap J^{(n^0,n^+,n^-)}_{\mathcal{O}_{K}^{\mathbb{C}}}$. Thus there exists $\tau_1\in(0,1)$ such that $\lambda_{n^+-n_0^++1}(\gamma_1(\tau_1))=r_1$. However, $r_1$ is not an eigenvalue for any boundary condition in $U_0\cap J^{(n^0,n^+,n^-)}_{\mathcal{O}_{K}^{\mathbb{C}}}$, which is a contradiction.
\end{proof}

From Lemma \ref{asymptotic behavoir for tilde A0}, we have shown asymptotic behavior  of the $n$-th eigenvalue near $\tilde{\mathbf{A}}_0$ from all the directions  in
${\mathcal{O}_{K}^{\mathbb{C}}}$.
  Next, we consider other boundary conditions   in  $J^{(n_0^0,n_0^+,n_0^-)}_{\mathcal{O}_{K}^{\mathbb{C}}}$.

\begin{lemma}\label{asymptotic behavior A0}
Let  $\pmb\omega_0=(P_0,Q_0,W_0)$ be given in (\ref{independent equation}). Then for any $\mathbf{B}_0\in J^{(n_0^0,n_0^+,n_0^-)}_{\mathcal{O}_{K}^{\mathbb{C}}}$ and for any path $\mathbf{B}_s\in J^{(n^0,n^+,n^-)}_{\mathcal{O}_{K}^{\mathbb{C}}}$, $s\in(0,1]$ such that $\mathbf{B}_s\to{\mathbf{B}}_0$ as $s\to0^+$, we have
\begin{align}\label{any boundary condition path d one-dimensional asymptotic behavior1 to A0}
\lim\limits_{s\to0^+}\lambda_n(\mathbf{B}_s)&=-\infty,\;n\leq n^+-n_0^+,\\\label{any boundary condition path d one-dimensional asymptotic behavior2 to A0}
\lim\limits_{s\to0^+}\lambda_n(\mathbf{B}_s)&=
\lambda_{n-(n^+-n_0^+)}({\mathbf{B}}_0),\;n >n^+-n_0^+.
\end{align}
\end{lemma}
\begin{proof}
 By Lemma \ref{JC path connected}, there exists a path $\tilde\gamma$  in $J^{(n_0^0,n_0^+,n_0^-)}_{\mathcal{O}_{K}^{\mathbb{C}}}$ to connect $\mathbf{B}_0$ and $\tilde{\mathbf{A}}_0$, which is given in (\ref{independent boundary condition}). $\lambda_1$ is continuous on $\tilde\gamma$ by Theorem \ref{omegatimesboundary condition continuity}. Fix any $r_1<\min\lambda_1(\tilde \gamma)$. If $n^+-n_0^+=0$, it suffices to show that
$\lambda_1$ is bounded from below on $\{ \mathbf{B}_s:s\in(0,1]\}$ by Lemma \ref{continuity principle}.
Suppose otherwise, there exists a sequence $\{s_n\}_{n=1}^\infty$ such that ${s_n}\to0^+$ as $n\to\infty$ and
\begin{align}\label{any boundary condition in JC contradiction1}
\lim\limits_{n\to\infty}\lambda_1(\mathbf{B}_{s_n})=-\infty.
\end{align}
For any given $\mathbf{A}_s$ in  Lemma \ref{asymptotic behavoir for tilde A0},
\begin{align}\label{any boundary condition in JC contradiction2}
\lim\limits_{s\to0^+}\lambda_1(\mathbf{A}_{s})=\lambda_1(\tilde{\mathbf{A}}_0).
\end{align}
Choose $r_{2,\tau}>r_1$ such that $\lambda_1(\tilde \gamma(\tau))$ with multiplicity  $m_{1,\tau}$ is the only eigenvalue in $(r_1,r_{2,\tau})$ and $r_{2,\tau}$ is not an eigenvalue of $(\pmb\omega_0,\tilde\gamma(\tau))$ for each $\tau\in[0,1]$. Then by Lemma \ref{local continuity of eigenvalues}, for any given $\tau\in[0,1]$, there exists a neighborhood $U_\tau$  of $\tilde\gamma(\tau)$ in $\mathcal{O}_K^{\mathbb{C}}$ such that for each $\mathbf{A}\in U_\tau$, $(\pmb\omega_0,\mathbf{A})$ has exactly $m_{1,\tau}$ eigenvalues in $(r_1,r_{2,\tau})$ and neither  $r_1$ nor $r_{2,\tau}$ is its eigenvalue.
Set $U:=\cup_{\tau\in[0,1]}U_{\tau}$. Then $r_1$ is not an eigenvalue of $(\pmb\omega_0,\mathbf{A})$ for any $\mathbf{A}\in U$. Since $\tilde\gamma$ is compact, it is easy to see that $U_{\varepsilon_0}:=\{\mathbf{B}\in\mathcal{O}_K^\mathbb{C}:\|S(\mathbf{B})-S(\tilde{\mathbf{A}})\|_{\mathcal{M}_{2d\times 2d}}<\varepsilon_0,\tilde{\mathbf{A}}\in\tilde\gamma\}\subset U$ for some $\varepsilon_0>0$.
It follows from (\ref{any boundary condition in JC contradiction1})--(\ref{any boundary condition in JC contradiction2}) that
there exist $n_1\in\mathbb{N}$ and $\tilde s\in(0,1]$ such that $\mathbf{B}_{s_{n_1}},\mathbf{A}_{\tilde s}\in U_{\varepsilon_0}\cap J^{(n^0,n^+,n^-)}_{\mathcal{O}_{K}^{\mathbb{C}}}$ and
\begin{align}\label{contradiction as a result}
\lambda_1(\mathbf{B}_{s_{n_1}})<r_1<\lambda_1(\mathbf{A}_{\tilde s}).
\end{align}
We can construct a path $\tilde\gamma_0$ in $U_{\varepsilon_0}\cap J^{(n^0,n^+,n^-)}_{\mathcal{O}_{K}^{\mathbb{C}}}$ to connect $\mathbf{B}_{s_{n_1}}$ and $\mathbf{A}_{\tilde s}$ as follows. Define $s_{ij}(\tilde\gamma_1(\tau)):=s_{ij}(\tilde\gamma(\tau))$ when $i\notin K$ or $j\notin K, \tau\in[0,1]$. Let $m_0=\sharp(K)$. Denote  $S_K(\tilde\gamma(\tau))=M^*_\tau S_K(\tilde{\mathbf{A}}_{0}) M_\tau$ to be the path to connect  $S_K(\mathbf{B}_{0})$ and $S_K(\tilde{\mathbf{A}}_{0})$ such that $M_\tau\in GL(m_0,\mathbb{C}),\tau\in[0,1]$.
Then define $S_K(\tilde\gamma_1(\tau)):=M^*_\tau S_K({\mathbf{A}}_{\tilde s}) M_\tau$.
Note that  ${\mathbf{A}}_{\tilde s}$ can be chosen sufficiently close to $\tilde{\mathbf{A}}_{0}$ such that $\tilde\gamma_1\subset U_{\varepsilon_0}\cap J^{(n^0,n^+,n^-)}_{\mathcal{O}_{K}^{\mathbb{C}}}$.
 Following Lemma \ref{Uvarepsilon conponent}, connect $\tilde\gamma_1(0)$ and  $\mathbf{B}_{s_{n_1}}$ by $\tilde\gamma_2$ in $U_{\varepsilon_0}\cap J^{(n^0,n^+,n^-)}_{\mathcal{O}_{K}^{\mathbb{C}}}$. Combining $\tilde \gamma_i$, $i=1,2$, we get the desired path $\tilde\gamma_0$.
 $\lambda_1(\tilde\gamma_0)$ is connected by Theorem  \ref{omegatimesboundary condition continuity}. Thus (\ref{contradiction as a result}) contradicts that $r_1$ is not an eigenvalue for any boundary condition in $\tilde\gamma_0$.

If $n^+-n_0^+>0$, for any $1\leq i\leq n^+-n_0^+$, assume that
$\lim\limits_{s\to0^+}\lambda_j(\mathbf{B}_s)=-\infty, \;\;1\leq j\leq i-1,$
we show that
\begin{align}\label{n=i+1 any boundary on JC}
\lim\limits_{s\to0^+}\lambda_{i}(\mathbf{B}_s)=-\infty.
\end{align}
Suppose otherwise, there exists a sequence $\{ \hat s^{(i)}_n\}_{n=1}^\infty$ such that ${ \hat s^{(i)}_n}\to 0^+$ as $n\to\infty$ and $\lambda_i$ is bounded from below on $\{\mathbf{B}_{\hat s^{(i)}_n}\}_{n=1}^\infty$. Then
$\lim\limits_{n\to\infty}\lambda_i(\mathbf{B}_{\hat s^{(i)}_n})=\lambda_1({\mathbf{B}}_0)$ by Lemmas \ref{continuity principle} and \ref{continuity-change-indices}. Note that
 $\lim\limits_{s\to0^+}\lambda_i({\mathbf{A}}_s)=-\infty$, where ${\mathbf{A}}_s$ is given in  Lemma \ref{asymptotic behavoir for tilde A0}.
Thus we can choose  $\mathbf{B}_{\hat s^{(i)}_{n_1}}, {\mathbf{A}}_{\hat s'_1}\in U_{\varepsilon_0}\cap J^{(n^0,n^+,n^-)}_{\mathcal{O}_{K}^{\mathbb{C}}}$ such that
\begin{align}\label{to be a contradiction for any boundary}
\lambda_i({\mathbf{A}}_{\hat s'_1})<r_1<\lambda_i(\mathbf{B}_{ \hat s^{(i)}_{n_1}})
\end{align}
for some  $n_1\in\mathbb{N}$ and some $\hat s'_1\in (0,1]$.
 Similar as above, there exists a path $\tilde\gamma_3$ in $U_{\varepsilon_0}\cap J^{(n^0,n^+,n^-)}_{\mathcal{O}_{K}^{\mathbb{C}}}$ to connect $\mathbf{B}_{\hat s_{n_1}}$ and $\mathbf{A}_{\hat s'_1}$. $\lambda_i(\tilde\gamma_3)$ is connected by Theorem  \ref{omegatimesboundary condition continuity}.
 Then the choice of $r_1$ contradicts (\ref{to be a contradiction for any boundary}).

 Suppose  (\ref{any boundary condition path d one-dimensional asymptotic behavior2 to A0}) is not true. Then we can find $\mathbf{B}_{ \hat s_{2}}, \mathbf{A}_{\tilde s_2}\in U_{\varepsilon_0}\cap J^{(n^0,n^+,n^-)}_{\mathcal{O}_{K}^{\mathbb{C}}}$ such that
 $\lambda_{n^+-n_0^++1}(\mathbf{A}_{\tilde s_2})>r_1>\lambda_{n^+-n_0^++1}(\mathbf{B}_{\hat s_2})$. By constructing $\tilde\gamma_4$ in $U_{\varepsilon_0}\cap J^{(n^0,n^+,n^-)}_{\mathcal{O}_{K}^{\mathbb{C}}}$ to connect $\mathbf{B}_{\hat s_2}$ and $\mathbf{A}_{\tilde s_2}$, one can get a contradiction as above.  This completes the proof.
\end{proof}

 Now, for any $\pmb\omega=(P,Q,W)\in\Omega$, we give  asymptotic behavior  of the $n$-th eigenvalue near any boundary condition in $ J^{(n_0^0,n_0^+,n_0^-)}_{\mathcal{O}_{K}^{\mathbb{C}}}, n_0^0\geq 1,$ from all the directions  in
${\mathcal{O}_{K}^{\mathbb{C}}}$.

\begin{lemma}\label{asymptotic behavior C0 any sturm liouville equation}
Let $\pmb\omega=(P,Q,W)\in\Omega$. Then for any $\mathbf{C}_0\in J^{(n_0^0,n_0^+,n_0^-)}_{\mathcal{O}_{K}^{\mathbb{C}}}$ and for any path $\mathbf{C}_s\in J^{(n^0,n^+,n^-)}_{\mathcal{O}_{K}^{\mathbb{C}}}$, $s\in(0,1]$ such that $\mathbf{C}_s\to{\mathbf{C}}_0$ as $s\to0^+$, we have
\begin{align}\label{any boundary condition path d one-dimensional asymptotic behavior1 to C0}
\lim\limits_{s\to0^+}\lambda_n(\mathbf{C}_s)&=-\infty,\;n\leq n^+-n_0^+,\\\label{any boundary condition path d one-dimensional asymptotic behavior2 to C0}
\lim\limits_{s\to0^+}\lambda_n(\mathbf{C}_s)&=
\lambda_{n-(n^+-n_0^+)}({\mathbf{C}}_0),\;n >n^+-n_0^+.
\end{align}
\end{lemma}
\begin{proof}
Let
\begin{align*}
\zeta(\tau):=(\tau P_0+(1-\tau)P,\tau Q_0+(1-\tau)Q,\tau W_0+(1-\tau)W),\;\;\tau\in[0,1],
\end{align*}
where $(P_0,Q_0,W_0)$ is given in (\ref{independent equation}). Then it is easy to verify that $\zeta(\tau)\in\Omega, \tau\in [0,1]$. $\zeta\times \mathbf{C}_0$ means $\{(\zeta(\tau),\mathbf{C}_0):\tau\in[0,1]\}$ and
choose $\hat r_1<\min\{\lambda_1(\zeta\times\mathbf{C}_0)\}$.

 Let $\hat r_{2,\tau}>\hat r_1$ such that $\lambda_1(\zeta(\tau),\mathbf{C}_0)$ with multiplicity $\hat m_{1,\tau}$ is the only eigenvalue in $(\hat r_1,\hat r_{2,\tau})$ and $\hat r_{2,\tau}$ is not an eigenvalue of $(\zeta(\tau),\mathbf{C}_0)$ for each $\tau\in[0,1]$.  Then by Lemma \ref{local continuity of eigenvalues}, there exists a neighborhood $\hat U_\tau$  of $(\zeta(\tau),\mathbf{C}_0)$, $\tau\in[0,1]$, in $\Omega\times\mathcal{O}_K^{\mathbb{C}}$ such that each $(\pmb\sigma,\mathbf{C})\in\hat U_\tau$ has exactly $\hat m_{1,\tau}$ eigenvalues in $(\hat r_1,\hat r_{2,\tau})$ and neither  $\hat r_1$ nor $\hat r_{2,\tau}$ is its eigenvalue.
Set $\hat U:=\cup_{\tau\in[0,1]}\hat U_{\tau}$. Then $\hat r_1$ is not an eigenvalue of $(\pmb\sigma,\mathbf{C})\in  \hat U$.
By the compactness of $\zeta$, there exists a neighborhood $U_{\mathbf{C}_0}$ of $\mathbf{C}_0$ in $\mathcal{O}_K^{\mathbb{C}}$ such that
$(\zeta(\tau), U_{\mathbf{C}_0})\subset \hat U$ for all $\tau\in[0,1]$.

If $n^+-n_0^+=0$, then by Lemma \ref{asymptotic behavior A0},
for any $\mathbf{B}_s$, $s\in(0,1],$   such that $\mathbf{B}_s\to \mathbf{C}_0$ as $s\to0^+$,  $\lambda_1$  is continuous on $\{(\pmb\omega_0,\mathbf{B}_s):s\in(0,1]\}\cup\{(\pmb\omega_0,\mathbf{C}_0)\}$. Thus there exists $s_0\in(0,1]$
 such that $\mathbf{B}_{s_0}\in U_{\mathbf{C}_0}\cap J^{(n^0,n^+,n^-)}_{\mathcal{O}_{K}^{\mathbb{C}}}$ and $\lambda_1(\pmb\omega_0,\mathbf{B}_{s_0})>\hat r_1$. Suppose there exists $s_1\in(0,1]$ such that $\mathbf{C}_{s_1}\in U_{\mathbf{C}_0}\cap J^{(n^0,n^+,n^-)}_{\mathcal{O}_{K}^{\mathbb{C}}}$ and $\lambda_1(\pmb\omega,\mathbf{C}_{s_1})<\hat r_1$.
In $(\Omega\times J^{(n^0,n^+,n^-)}_{\mathcal{O}_{K}^{\mathbb{C}}})\cap \hat U$,
via constructing a path to connect $(\pmb\omega_0,\mathbf{B}_{s_0})$ and $(\pmb\omega_0,\mathbf{C}_{s_1})$ by Lemma \ref{Uvarepsilon conponent}  then a path to connect $(\pmb\omega_0,\mathbf{C}_{s_1})$ and $(\pmb\omega,\mathbf{C}_{s_1})$ through $\zeta$,
we get a path $\hat \gamma$  to connect $(\pmb\omega_0,\mathbf{B}_{s_0})$ and  $(\pmb\omega,\mathbf{C}_{s_1})$.
$\lambda_1(\hat\gamma)$ is connected  by Theorem \ref{omegatimesboundary condition continuity}. This contradicts  $\lambda_1(\pmb\omega,\mathbf{C}_{s_1})<\hat r_1<\lambda_1(\pmb\omega_0,\mathbf{B}_{s_0})$ and $\hat r_1$ is not an eigenvalue of
$(\pmb\sigma,\mathbf{C})\in  \hat \gamma\subset \hat U$.

Let $n^+-n_0^+>0$. Then for any $1\leq i \leq n^+-n_0^+$, assuming that we have shown that
$\lim\limits_{s\to0^+}\lambda_j(\pmb\omega,\mathbf{C}_s)=-\infty, 1\leq j\leq i-1$,
we now show that $\lim\limits_{s\to0^+}\lambda_i(\pmb\omega,\mathbf{C}_s)=-\infty$.
If it is not true, there exists $\{s_{n}\}_{n=1}^\infty$ such that  $\lambda_i$ is bounded from below on  $\{(\pmb\omega,\mathbf{C}_{s_n})\}_{n=1}^{\infty}$ and thus $\lim\limits_{n\to\infty}\lambda_{i}(\pmb\omega,\mathbf{C}_{s_n})=\lambda_1(\pmb\omega,\mathbf{C}_{0})$ by Lemmas \ref{continuity principle} and \ref{continuity-change-indices}. This, together with Lemma \ref{asymptotic behavior A0}, implies that
there exists  $s_2\in(0,1]$ and $n_1\in\mathbb{N}$ such that $\mathbf{B}_{s_2},\mathbf{C}_{s_{n_1}}\in U_{\mathbf{C}_0}\cap J^{(n^0,n^+,n^-)}_{\mathcal{O}_{K}^{\mathbb{C}}}$ and
\begin{align}\label{contradiction as a result different sturm liouville equation}
\lambda_i(\pmb\omega_0,\mathbf{B}_{s_2})<\hat r_1<\lambda_i(\pmb\omega, \mathbf{C}_{s_{n_1}}),
\end{align}
where $\mathbf{B}_{s}\in J^{(n^0,n^+,n^-)}_{\mathcal{O}_{K}^{\mathbb{C}}}, s\in(0,1]$ is any path such that
$\mathbf{B}_{s}\to \mathbf{C}_0$ as $s\to 0^+$. We can construct a path $\hat \gamma_1$ in  $(\Omega\times J^{(n^0,n^+,n^-)}_{\mathcal{O}_{K}^{\mathbb{C}}})\cap\hat U$ to connect $(\pmb\omega_0,\mathbf{B}_{s_2})$ and  $(\pmb\omega,\mathbf{C}_{s_{n_1}})$ as above.
$\lambda_i(\hat\gamma_1)$ is connected by Theorem \ref{omegatimesboundary condition continuity}. Thus (\ref{contradiction as a result different sturm liouville equation}) contradicts that $\hat r_1$ is not an eigenvalue for any $(\pmb\sigma,\mathbf{C})\in\hat\gamma_1$. This completes the proof of (\ref{any boundary condition path d one-dimensional asymptotic behavior1 to C0}).

 (\ref{any boundary condition path d one-dimensional asymptotic behavior2 to C0}) can be shown similarly as above and the proof is complete.
\end{proof}
To conclude this section, we combine the Sturm-Liouville equations and boundary conditions to get the following asymptotic behavior of the $n$-th eigenvalue:

\begin{theorem}\label{pmb-omega-mathbf{A}}
For any $(\pmb\omega,\mathbf{A})\in\Omega\times J^{(n_0^0,n_0^+,n_0^-)}_{\mathcal{O}_{K}^{\mathbb{C}}}$ and for any path $(\pmb\omega_s,\mathbf{A}_s)\in \Omega\times J^{(n^0,n^+,n^-)}_{\mathcal{O}_{K}^{\mathbb{C}}}$, $s\in(0,1]$ such that $(\pmb\omega_s,\mathbf{A}_s)\to(\pmb\omega,\mathbf{A})$ as $s\to0^+$, we have
\begin{align}\label{any boundary condition path d one-dimensional asymptotic behavior1 to omega A1}
\lim\limits_{s\to0^+}\lambda_n(\pmb\omega_s,\mathbf{A}_s)&=-\infty,\;n\leq n^+-n_0^+,\\\label{any boundary condition path d one-dimensional asymptotic behavior2 to omega A2}
\lim\limits_{s\to0^+}\lambda_n(\pmb\omega_s,\mathbf{A}_s)&=
\lambda_{n-(n^+-n_0^+)}(\pmb\omega,\mathbf{A}),\;n >n^+-n_0^+.
\end{align}
\end{theorem}
\begin{proof}
The proof is also by a contradiction argument as that used in the proof of  Lemmas \ref{asymptotic behavoir for tilde A0}, \ref{asymptotic behavior A0} and \ref{asymptotic behavior C0 any sturm liouville equation}. So we just sketch the proof here. It follows from  Lemma \ref{asymptotic behavior C0 any sturm liouville equation} that for the path $(\pmb\omega,\mathbf{A}_s)$, we have the asymptotic behavior (\ref{any boundary condition path d one-dimensional asymptotic behavior1 to omega A1})--(\ref{any boundary condition path d one-dimensional asymptotic behavior2 to omega A2}) with $\pmb\omega_s$ replaced by $\pmb\omega$. Thus, to get the contradiction, it suffices to construct a path  in $(\Omega\times J^{(n^0,n^+,n^-)}_{\mathcal{O}_{K}^{\mathbb{C}}})\cap U$ to connect $(\pmb\omega,\mathbf{A}_{s_1})$ and $(\pmb\omega_{s_2},\mathbf{A}_{s_2})$ for some $s_1,s_2\in(0,1]$, where
$U$ is a sufficiently small neighborhood of  $(\pmb\omega,\mathbf{A})$ in $\Omega\times\mathcal{O}_K^\mathbb{C}$. This can be obtained due to Lemma \ref{Uvarepsilon conponent}.
\end{proof}

\section{Derivative formulas and comparison  of  eigenvalues }\label{Derivative formulas and comparison  of  eigenvalues with respect to boundary conditions}

In this section, we give derivative formula of a continuous simple eigenvalue branch with respect to  boundary conditions. Then we obtain some inequalities of eigenvalues.

\begin{theorem}\label{Directional derivative of eigenvalue}
 Fix $\pmb\omega\in\Omega$. Let $\lambda_*(\mathbf{A})$ be a simple eigenvalue for $\mathbf{A}$, $y$ be a normalized eigenfunction for $\lambda_*(\mathbf{A})$,  and $\Lambda$ be a continuous simple
eigenvalue branch defined on a neighborhood of $\mathbf A$ in $\mathcal O_{K}^{\mathbb C}$
through $\lambda_*$ for some $K\subset\{1,\cdots,2d\}$.
Then $\Lambda$ is differentiable at $\mathbf{A}$ and its Frechet derivative formula is
given by
\begin{align}\label{derivative formula}
d\Lambda\big|_{\mathbf A}(H)=u^*H{u}
\end{align}
for any $2d\times 2d$ Hermitian matrix $H$, where $u=(u_1,\cdots,u_{2d})^T$ is given by
\begin{align*}
u_i=\left\{
\begin{aligned}
\rho_{i},&\ i\in \{1,\cdots,2d\}\setminus K,\\ \eta_{i},&\ i\in K,
\end{aligned}\right.
\end{align*}
and
$
(\rho_{1},\cdots,\rho_{2d},\eta_{1},\cdots,\eta_{2d}):=(-y(a)^T,y(b)^T,(Py'(a))^T,(Py'(b))^T).
$
\end{theorem}
\begin{proof}
Let $\mathbf A:=[A(S)\;|\;B(S)]$ and $\mathbf B:=[A(S+H)\;|\;B(S+H)]\in \mathcal{O}^\mathbb{C}_K$,  which correspond to Hermitian matrices $S$ and $S+H$, respectively.   Then there exists an eigenfunction $\tilde y=y_{\Lambda(\mathbf{B})}$ for $\Lambda(\mathbf{B})$ such that $P\tilde y'\to Py'$ as $\mathbf{B}\to \mathbf{A}$ in the sense of Lemma \ref{continuous choice of eigenfunctions}.

Note that  $\tilde y$ and ${y}$ satisfy
\begin{align*}
-(P\tilde{y}')'+Q\tilde{y}=\lambda(\mathbf{B})W\tilde{y},\;\;
-(P y')'+Q y=\lambda(\mathbf{A})W y.
\end{align*}
Thus
\begin{align*}
(\lambda(\mathbf{B})-\lambda(\mathbf{A}))\int_a^b(W\tilde y, y)_ddt
=\left(\begin{pmatrix}0&-I_{2d}\\I_{2d}&0\end{pmatrix}\tilde Y(a,b),
Y(a,b)\right)_{4d}.
\end{align*}
Denote
\begin{align*}
Y:=Y(a,b),\;\;
X:=\begin{pmatrix}u\\v\end{pmatrix},
\end{align*}
where $v:=(v_1,\cdots,v_{2d})^T$ is given by
\begin{align*}
v_i=\left\{
\begin{aligned}
 \eta_{i},&\ i\in\{1,\cdots,2d\}\setminus K,\\-\rho_{i},&\ i\in K.
\end{aligned}\right.
\end{align*}
$\tilde Y, \tilde{X}, \tilde u, \tilde v$ can be defined similarly.
Then for $\mathbf A$ and $\mathbf B$,
we have
\begin{align*}
[A(S)\;|\;B(S)]Y=&[S\;|\;I_{2d}]X\\
=&[S\;|\;I_{2d}]\begin{pmatrix}u\\v\end{pmatrix}\\
=&Su+v=0,\\
[A(S+H)\;|\;B(S+H)]\tilde Y=&[S+H\;|\;I_{2d}]\tilde X\\
=&[S+H\;|\;I_{2d}]\begin{pmatrix}\tilde u\\ \tilde v\end{pmatrix}\\=&(S+H)\tilde u+\tilde v=0.
\end{align*}
It follows that
\begin{align*}
(\lambda(\mathbf{B})-\lambda(\mathbf{A}))\int_a^b(W\tilde y, y)_ddt
=&Y^*\begin{pmatrix}0&-I_{2d}\\I_{2d}&0\end{pmatrix}\tilde{Y}\\
=&\sum_{i=1}^{2d}(\bar{\eta}_i\tilde\rho_i-\bar{\rho}_i\tilde\eta_i)\\
=&\sum_{i\in K}(-\bar u_i\tilde v_i+\bar v_i\tilde u_i)+\sum_{i\notin K}(\bar v_i\tilde u_i-\bar u_i\tilde v_i)\\
=&\sum_{i=1}^{2d}(\bar v_i\tilde u_i-\bar u_i\tilde v_i)\\
=&X^*\begin{pmatrix}0&-I_{2d}\\I_{2d}&0\end{pmatrix}\tilde{X}\\
=&(u^*,v^*)\begin{pmatrix}0&-I_{2d}\\I_{2d}&0\end{pmatrix}\begin{pmatrix}\tilde{u}\\ \tilde{v}\end{pmatrix}\\
=&(u^*,-u^*S)\begin{pmatrix}0&-I_{2d}\\I_{2d}&0\end{pmatrix}\begin{pmatrix}\tilde{u}\\ -(S+H)\tilde{u}\end{pmatrix}\\
=&u^*H\tilde{u}.
\end{align*}
Letting $H\to0$, we get (\ref{derivative formula}).
\end{proof}
The positive direction of continuous eigenvalue branches is given in the following result, which is a direct consequence of Theorems \ref{ equivalence of three multiplicities of an eigenvalue} and \ref{Directional derivative of eigenvalue}.

\begin{corollary}\label{corollary on Directional derivative of eigenvalue}
 Fix $\pmb\omega\in\Omega$. Let $\lambda_*(\mathbf{A})$ be an eigenvalue for $\mathbf{A}$  and $\Lambda$ be a continuous
eigenvalue branch defined on a neighborhood $U$ of $\mathbf A$ in $\mathcal O_{K}^{\mathbb C}$
through $\lambda_*$.
Then $$\Lambda(\mathbf{A})\leq\Lambda(\mathbf{B})$$
if $\mathbf{B}\in U$ and $S(\mathbf{B})-S(\mathbf{A})$ is non-negative.
\end{corollary}

\begin{remark}
 Theorem \ref{Directional derivative of eigenvalue} (it is required that $\lambda_n(\mathbf{A})$ is simple) and Corollary \ref{corollary on Directional derivative of eigenvalue}  also hold true when $\lambda_n$ is continuous  on a neighborhood $U$ of $\mathbf{A}$ in $\mathcal O_{K}^{\mathbb C}$.
\end{remark}

\bigskip

\noindent\textbf{Acknowledgements}\bigskip

H. Zhu sincerely thanks  Prof. Y. Long for his instructive advice. X. Hu is supported in part  by NSFC (No. 11425105, 11790271). H. Zhu is supported in part by CPSF (No. 2018M630266).

\bigskip

\noindent
Email addresses: \newline xjhu@sdu.edu.cn (X. Hu), \newline
201511242@mail.sdu.edu.cn (L. Liu),
\newline 201790000005@sdu.edu.cn (L. Wu),\newline haozhu@nankai.edu.cn (H. Zhu).


\begin{thebibliography}{99}
\bibitem{Arnold1967} V. I. Arnold, Characteristic class entering in quantization conditions, Funct. Anal.  Appl. 1 (1967)  1--13.
\bibitem{Arnold2000} V. I. Arnold, The complex Lagrangian Grassmannian, (Russian)
Funktsional. Anal. i Prilozhen. 34 (2000)  63--65; translation in
Funct. Anal. Appl. 34 (2000) 208--210.

\bibitem{Cao2007} X. Cao, Q. Kong, H. Wu, A. Zettl, Geometric aspects of Sturm-Liouville problems. III. Level surfaces of the n-th eigenvalue, J. Comput. Appl. Math.  208  (2007)   176--193.
\bibitem{Conway1986} J. Convey, Functions of One Complex Variable, 2nd ed., Springer-Verlag, Berlin/New York, 1986.
\bibitem{Courant1953} R. Courant, D. Hilbert, Methods of Mathematical Physics, Interscience Publishers, New York, 1953.
\bibitem{Eastham1999} M. S. P. Eastham, Q. Kong, H. Wu, A. Zettl, Inequalities among eigenvalues of Sturm-Liouville problems, J. Inequalities Appl. 3 (1999) 25--43.
 \bibitem{Everitt1997} W. N.  Everitt, M. M\"{o}ller, A. Zettl, Discontinuous dependence of the $n$-th Sturm-Liouville  eigenvalues,
 in ``General Inequalities", Birkhauser, Basel, 1997.

\bibitem{Fu2012} S. Fu, Z. Wang, The relationships among multiplicities of a $J$-self-adjoint differential operator's eigenvalue, Pac. J. Appl. Math. 4 (2012) 292--303.
\bibitem{Fu2014} S. Fu, Z. Wang, Relationships among three multiplicities of a differential operator's eigenvalue, Applied Math. 5 (2014) 2185--2194.
\bibitem{Horn2013} R. A. Horn, C. R. Johnson, Matrix Analysis, 2nd ed., Cambridge University  Press, Cambridge, 2013.
\bibitem{Hu2011}   X. Hu, P. Wang, Conditional Fredholm determinant for the S-periodic orbits in Hamiltonian systems, J. Funct. Anal.  261  (2011)  3247--3278.
  \bibitem{Hu2016} X. Hu, P. Wang, Eigenvalue problem of Sturm-Liouville systems with seperated boundary conditions, Math. Z. 283 (2016) 339--348.

 \bibitem{Kato1984} T. Kato, Perturbation Theory for Linear Operators, 2nd ed., Springer-Verlag, Berlin/Heidelberg/New York/Tokyo, 1984.
\bibitem{Kong1999} Q. Kong, H. Wu, A. Zettl, Dependence of the $n$-th Sturm-Liouville eigenvalue on the problem, J. Differential
Equations 156 (1999) 328--354.

\bibitem{Kong2000} Q. Kong, H. Wu, A. Zettl, Geometric aspects of Sturm-Liouville
problems, I. Structures on spaces of boundary conditions,
Proc. Roy. Soc. Edinburgh Sect. A 130 (2000) 561--589.

\bibitem{Kong2004} Q. Kong, H. Wu, A. Zettl, Multiplicity of Sturm-Liouville eigenvalues,
J. Comput. Appl. Math. 171 (2004) 291--309.

\bibitem{Kong1996} Q. Kong, A. Zettl, Eigenvalues of regular Sturm-Liouville problems, J. Differential Equations 131 (1996)
1--19.
\bibitem{Kong19962}Q. Kong, A. Zettl, Dependence of eigenvalues of Sturm-Liouville problems on the boundary, J. Differential Equations  126  (1996)  389--407.
\bibitem{Liu2018}
L. Liu, Lagrangian Grassmanian Manifold $\Lambda(2)$, Front. Math. China (2018), doi: 10.1007/s11464-018-0683-2.
\bibitem{Long1991}Y. Long,  The structure of the singular symplectic matrix set, Sci. China Ser. A  34  (1991)  897--907.
\bibitem{Long2002}Y. Long, Index Theory for Symplectic Paths with Applications, Progr.  Math., vol. 207, Birkh\"{a}user, Basel, 2002.
\bibitem{Naimark1968}
M. A. Naimark, Linear Differential Operators, Ungar, New York, 1968.
\bibitem{Rellich1942}F. Rellich,
St\"{o}rungstheorie der Spektralzerlegung V, Math. Ann. 118 (1942)  462--484.
\bibitem{Rellich1950}
F. Rellich, St\"{o}rungstheorie der Spektralzerlegung, Proceedings of the International Congress of Mathematicians 1 (1950) 606--613.
\bibitem{Shi2010}
D. Shi, Z. Huang, Relationships of multiplicities of eigenvalues of a higher-order ordinary differential operator,
Acta Math. Sinica (Chin. Ser.)  53  (2010) 763--772.
\bibitem{Wang2008}
 A. Wang, J. Sun, A. Zettl,  The classification of self-djoint boundary conditions: separated, coupled, and mixed, J. Funct. Anal.  255  (2008)  1554--1573.
 \bibitem{Wang2005}
 Z. Wang, H. Wu, Equalities of multiplicities of a Sturm-Liouville eigenvalue, J. Math. Anal. Appl.  306  (2005) 540--547.
 \bibitem{Weidmann1987} J. Weidmann, Spectral Theory of Ordinary Differential Operators, Lecture Notes in
Mathematics, vol. 1258, Springer-Verlag, Berlin, 1987.
\bibitem{Zettl1997} A. Zettl, Sturm-Liouville problems, In ``Spectral Theory and Computational Methods of Sturm-Liouville Problems" (ed. Hinton and P. Schaefer) (New York: Dekker, 1997).
\bibitem{Zettl2005} A. Zettl, Sturm-Liouville Theory, Mathematical Surveys Monographs, vol. 121, Amer. Math. Soc., 2005.

\bibitem{Zhu2017} H. Zhu,  Y. Shi, Dependence of eigenvalues on the boundary conditions of Sturm-Liouville problems with one singular endpoint, J. Differential Equations 263 (2017)  5582--5609.
\end{thebibliography}
\end{document}